\documentclass[a4paper]{article}
\usepackage{amsfonts}
\usepackage{amssymb,amsthm, amsmath,graphicx,bm,ebezier,amsthm}
\usepackage{hyperref}
\usepackage{caption}
\usepackage[normalem]{ulem}
\usepackage{url,lineno}
\usepackage[utf8]{inputenc}
\usepackage{arydshln}
\usepackage[ruled,vlined,linesnumbered]{algorithm2e}
\usepackage[noend]{algpseudocode}
\usepackage{lscape}
\usepackage{multirow}
\usepackage[dvipsnames]{xcolor}
\usepackage{stmaryrd}
\usepackage{todonotes}
\usepackage{tikz-qtree}

\newtheorem{theorem}{Theorem}

\newtheorem{proposition}[theorem]{Proposition}
\newtheorem{corollary}[theorem]{Corollary}
\newtheorem{lemma}[theorem]{Lemma}

\theoremstyle{remark}

\newtheorem{example}[theorem]{Example}
\newtheorem{remark}[theorem]{Remark}

\def\CaC{\mathcal{C}}
\def\CaH{\mathcal{H}}

\title{Positioned and primary positioned $\CaC$-semigroups}

\date{}
\author{
C. Cisto and
R. Tapia-Ramos
}

\begin{document}

\maketitle

\abstract{
Let $\CaC$ be a positive integer cone and $k\in \CaC$. A $\mathcal{C}$-semigroup $S$ is $k$-positioned if for every $h\in \mathcal{C}\setminus S$ we have that $k-h$ belongs to $S$. In this work, we focus on this family of semigroups and introduce primary positioned $\mathcal{C}$-semigroups, characterizing a subfamily of them through the perspective of irreducibility. Furthermore, we provide some procedures to compute all such semigroups, describing a family of graphs containing all the primary positioned $\mathcal{C}$-semigroups for a fixed $k\in \CaC$.
}

\medskip
{\small

}

\section*{Introduction}

Let $\mathbb{N}$ be the set of non-negative integers and denote by $\mathbb{Q}_{\geq 0}$ the set of non-negative rational numbers. Considering a finite set $A\subseteq \mathbb{N}^d$ (for a non-zero natural number $d$), the set 
\[
\mathrm{cone}(A)=\left\{\sum_{i=1}^n q_i a_i\mid n\in \mathbb{N}, a_i\in A, q_i\in \mathbb{Q}_{\geq 0}\text{ for all }i\in \{1,\ldots,n\}\right\}\]
is the \textit{rational affine cone} spanned by $A$. We say that $\CaC=\mathrm{cone}(A)\cap \mathbb{N}^d$ is the \textit{positive integer affine cone} spanned by $A$. See \cite{bruns2009polytopes} for a detailed survey on this topic.
From now on, we omit the word ``affine'', and assume that every cone we introduce throughout the work is affine.  Let $S$ be an additive submonoid of $\mathbb{N}^d$. We say that a subset $B\subseteq S$ generates $S$ if 
\[
S=\left\{\sum_{i=1}^n \lambda_i b_i\mid n,\lambda_i \in \mathbb{N}, b_i\in B \text{ for all }i\in \{1,\ldots,n\}\right\}.
\]
If no proper subset of $B$ generates $S$, then $B$ is called a \textit{minimal system of generators of $S$}. The monoid $S$ is called an \textit{affine semigroup} if there exists a finite set $B\subset S$ that generates it. In \cite{rosales1999finitely}, it is established that every affine semigroup has a unique minimal system of generators, and we denote it by $\operatorname{msg}(S)$. Given a positive integer cone $\CaC$, an affine semigroup $S\subseteq\CaC$ is called a $\CaC$-semigroup if the set $\CaC\setminus S$ is finite. For further details on systems of generators of a $\CaC$-semigroup, see also \cite{cisto-cofinite,diaz2022characterizing}.

The notion of $\CaC$-semigroup has been introduced in \cite{Wilf-Vigneron} as a generalization of the concept of \textit{numerical semigroup}. 
Recall that a numerical semigroup is a submonoid of $\mathbb{N}$ having a finite complement in it. In particular, this coincides with that of a $\CaC$-semigroup where $\CaC=\mathbb{N}$. Another special case is when $\CaC=\mathbb{N}^d$ (for $d>1$), which has been considered first in \cite{failla2016algorithms} and where this kind of monoids are called \textit{generalized numerical semigroups} (GNS for short). 
Numerical semigroups have been extensively studied in the literature, with notable references such as \cite{libroRosales} and the references therein.
A recent research direction has been developed to generalize properties and results, regarding numerical semigroups, in the more general context of $\CaC$-semigroups or GNSs.

Throughout this paper, if $h$ is a vector in $\mathbb{N}^d$ for $d>1$, we denote by $h^{(i)}$ its $i$-th coordinate. Moreover, we use $0$ to denote both the zero vector and zero integer, since there is no risk of confusion. Let $S\subseteq \mathbb{N}^d$ be a $\CaC$-semigroup. We denote by $\CaH(S)=\CaC\setminus S$, which is called the \textit{set of gaps} of $S$, and its cardinality, referred to as the \textit{genus} of $S$, is denoted by $\operatorname{g}(S)$. These definitions naturally arise from the analogous concepts in the framework of numerical semigroups. 
A key difference between the cases $d=1$ and $d>1$ is the existence of a natural total order on $\mathbb{N}$.
If $d=1$, some important invariants of a numerical semigroup strongly depend on the total order. Among them, we recall the \textit{Frobenius number}, which is $\operatorname{F}(S)=\max (\CaH(S))$ (conventionally $\operatorname{F}(\mathbb{N})=-1$), the \textit{conductor}, that is $\operatorname{c}(S)=\operatorname{F}(S)+1$, and the \textit{multiplicity}, which is $\operatorname{m}(S)=\min(S\setminus \{0\})$. 
In the case $d>1$, a generalization of the Frobenius number can be considered by fixing a term order on $\mathbb{N}^d$. Recall that a \textit{term order} (or monomial order) $\preceq$ on $\mathbb{N}^d$ is a total order satisfying $0\preceq v$ for all $v\in \mathbb{N}^d$ and if $u\preceq v$ for $u,v\in \mathbb{N}^d$, then $u+w\preceq v+w$ for all $w\in \mathbb{N}^d$. 
A detailed  discussion of orders can be found in \cite{cox1997ideals}.
Following the definitions in
\cite{failla2016algorithms,Wilf-Vigneron}, we define the \textit{Frobenius element} of $S$ with respect to $\preceq$ as the vector $\operatorname{F}_\preceq (S)=\max_\preceq \CaH(S)$ (conventionally $\operatorname{F}_\preceq(\mathbb{N}^d)=(-1,\ldots,-1)$).

The arguments developed in this paper are inspired by the concept of \textit{positioned numerical semigroup}, introduced by \cite{Branco-Faria-Rosales} as an extension of the notion of symmetric numerical semigroup. In particular, given $S$ a numerical semigroup and $k\in \mathbb{N}$, $S$ is called $k$-positioned if for all $h\in \CaH(S)$, we have $k-h\in S$. Furthermore, $S$ is said to be \textit{positioned} if it is $\operatorname{F}(S)+\operatorname{m}(S)$-positioned. The same authors obtained many properties and results on this family of numerical semigroups in different articles, for instance, the one mentioned above and \cite{rosales2022positioned2, rosales2022positioned1} (see also the thesis \cite{faria2022dense} for an overview). 
Drawing inspiration from some results presented in the mentioned papers, in our work, given a positive integer cone $\CaC$, and $k\in \CaC$, we introduce the notion of $k$-positioned for $\CaC$-semigroups and explore some properties of it. 

The paper is structured more in detail as follows. In Section \ref{Sec1}, we provide some notations and recall the notion of symmetric and pseudo-symmetric $\CaC$-semigroups, along with their main properties, which will be useful throughout our work. In Section \ref{Sec2}, we consider the definition of $k$-positioned $\CaC$-semigroup for $k\in \CaC$. The study of this notion led us to introduce the concept of \textit{primary positioned} $\CaC$-semigroups. In particular, they are $k$-positioned $\CaC$-semigroups in which a particular relation between some invariants of the semigroup is
satisfied. These invariants are obtained by considering a different interpretation of the concept of conductor and the multiplicity of a numerical semigroup. Additionally, we find that for the genus of every $k$-positioned $\CaC$-semigroup an upper bound depending on $k$. In Section \ref{Sec3}, we show that if this bound on the genus is sharp, then the $\CaC$-semigroup is primary positioned. Moreover, we characterize this family of $\CaC$-semigroups by the notions of UESY-semigroups and PEPSY-semigroups, which are related to the concept of symmetric and pseudo-symmetric. In Section \ref{Sec4}, fixed a positive integer cone $\CaC$, we study
for which vectors $k\in \CaC$ there exists a primary positioned $\CaC$-semigroup for $k$. Furthermore, we present a complete characterization of generalized numerical semigroups. In Section \ref{section-tree}, we show how it is possible to arrange the set of all primary positioned $\CaC$-semigroups for a fixed $k\in \CaC$ in a family of rooted trees, providing also some properties about this construction. This fact suggests an algorithmic procedure to generate the set of all primary positioned $\CaC$-semigroups for a fixed $k\in \CaC$. This procedure is summarized in Section \ref{Sec6}, where some examples of this construction are also given. Finally, we mention that in some parts of the work, we support our results by also providing examples and figures, for which we have used
third-party software \texttt{Normaliz} \cite{normaliz},  and \texttt{Mathematica}\cite{mathematica}. Moreover, the \texttt{GAP}\cite{GAP} package \texttt{numericalsgps}\cite{NumericalSgps} gave us a great help.

\section{Background}\label{Sec1}

This section is devoted to introducing several preliminaries. Specifically, we define the concepts of symmetric and pseudo-symmetric $\CaC$-semigroups. Additionally, we present some well-known results related to these definitions, which will be useful in the upcoming sections.

Let $S\subseteq \mathbb{N}^d$ be a $\CaC$-semigroup. A gap $x\in\CaH(S)$ of $S$  is called a \textit{pseudo-Frobenius element of $S$} if $ x + S \setminus \{0\} \subseteq S$. The set of pseudo-Frobenius elements of $S$ is denoted by $\operatorname{PF}(S)$. Let $\preceq$ be a term order on $\mathbb{N}^d$. We say that $S$ is \textit{symmetric} if $\operatorname{PF}(S)=\{\operatorname{F}_\preceq(S)\}$. Similarly, we say that $S$ is \textit{pseudo-symmetric} if $\operatorname{PF}(S)=\{\operatorname{F}_\preceq(S), \frac{\operatorname{F}_\preceq(S)}{2}\}$. A semigroup $S$ is called \textit{irreducible} if it is either symmetric or pseudo-symmetric. This definition corresponds to \cite[Theorem 21 and Proposition 23]{pseudo-Fb}.

This work considers different orders on $\mathbb{N}^d$, apart from a term order $\preceq$ on $\mathbb{N}^d$. Given a non-empty submonoid $A \subseteq \mathbb{N}^d$ and elements $x, y \in \mathbb{N}^d$, the partial order induced by $A$ is defined by, $x\leq_A y$  if and only if $y - x \in A$. The partial order induced by a numerical semigroup has been widely studied (see \cite{Delgado-Garcia-Robles, Chomp, libroRosales}). Furthermore,  the partial order induced by a $\CaC$-semigroup has also been discussed in the literature (see \cite{Cisto-unboundedness, pseudo-Fb, WilfMarinTapia}). Building on these induced orders, and drawing inspiration from \cite{garcia-on-some} for any $k\in \CaC$  we defined the sets,
\[\operatorname{I}_\CaC(k)= \{ x \in \CaC \mid x\leq_\CaC k\}, \quad \operatorname{I}_S(k) = \{ x \in S \mid x\leq_\CaC k\}.\]
It is pointed out that our definition of $\operatorname{I}_\CaC(k)$ coincides with the definition of $\mathcal{B}(k)$ provided in \cite{Rosales-Tapia-Vigneron}.

From now on, for any set $X$, the symbol $|X|$ denotes the cardinality of the set $X$. Observe that, for any $k\in\CaC$ such that $\CaH(S)\subseteq \operatorname{I}_\CaC(k)$, it holds that $|\operatorname{I}_\CaC(k)|=\operatorname{g}(S)+|\operatorname{I}_S(k)|$.

In the following statements, we show the main characterizations of symmetric and pseudo-symmetric $\CaC$-semigroup, along with the required references, respectively.

\begin{proposition}\label{caracSymm}
Let $S\subseteq \mathbb{N}^d$ be a $\CaC$-semigroup. Then  the following statements are equivalent conditions for $S$ to be symmetric: 
\begin{enumerate}
    \item (\cite[Proposition 3]{garcia-on-some}) $\operatorname{g}(S)=|\operatorname{I}_S(\operatorname{F}_\preceq (S))|$ for some term order $\preceq$ in $\mathbb{N}^d$.
    \item (\cite[Theorem 3.6]{bhardwaj2023affine}) For some term order $\preceq$ on $\mathbb{N}^d$, it holds $\operatorname{F}_\preceq (S)-h\in S$ for all $h\in \CaH(S)$.
\end{enumerate} 
\end{proposition}

\begin{proposition}\label{caracPesudo}
    Let $S\subseteq \mathbb{N}^d$ be a $\CaC$-semigroup. Then  the following statements are equivalent conditions for $S$ to be pseudo-symmetric: 
\begin{enumerate}
    \item (\cite[Proposition 4]{garcia-on-some}) $\operatorname{g}(S)=1+|\operatorname{I}_S(\operatorname{F}_\preceq(S))|$ and $\frac{\operatorname{F}_\preceq(S)}{2}\in \mathbb{N}^d$, for some term order $\preceq$ on $\mathbb{N}^d$.
    \item (\cite[Theorem 3.7]{bhardwaj2023affine}) For some term order $\preceq$ on $\mathbb{N}^d$, it holds $\operatorname{F}_\preceq (S)-h\in S$ for all $h\in \CaH(S)\setminus\left\{\frac{\operatorname{F}_\preceq (S)}{2}\right\}$.
\end{enumerate} 
\end{proposition}

For any $\CaC$-semigroup $S$, an element $x \in \CaH(S)$ is called a \textit{special gap} of $S$ if $x \in \operatorname{PF}(S)$ and $2x \in S$. The set of all special gaps of $S$ is denoted by $\operatorname{SG}(S)$. Observe that if $S$ is irreducible, then $\operatorname{SG}(S)=\{\operatorname{F}_\preceq (S)$\}, for some term order $\preceq$. Furthermore, let us denote $\mathcal{F}(S)=\{f\in \CaC \mid f=\operatorname{F}_\preceq(S) \text{ for some term order} \preceq$\}.

\begin{proposition}
    \cite[Proposition 4.3]{Cisto-unboundedness}
    Let $S$ be a $\CaC$-semigroup. Then 
    \[
    \mathcal{F}(S)\subseteq \operatorname{Maximals}_{\leq_\CaC}(\CaH(S))\subseteq\operatorname{SG}(S)\subseteq\operatorname{PF}(S).\]
\end{proposition}

As a consequence of the above result, we present the following.

\begin{corollary} \label{unique-Fb}
    Let $S$ be a $\CaC$-semigroup. The following holds:
    \begin{enumerate}
        \item If $|\operatorname{Maximals}_{\leq_\CaC}(\CaH(S))|=1$, then the Frobenius element is the same for every term order. 
        \item If $S$ is irreducible, then $\mathrm{Maximals}_{\leq_\CaC}\CaH(S)=\{\operatorname{F}_\preceq(S)\}$ for every term order $\preceq$.
    \end{enumerate} 
\end{corollary}


In the case where $S$ is irreducible $\CaC$-semigroup, the Frobenius element of $S$, as stated in the above corollary, is independent of the chosen term order. Therefore,  from this point forward, when there is no risk of ambiguity or misunderstanding, we will use the notation $\operatorname{F}(S)$ instead of $\operatorname{F}_{\preceq}(S)$.

\section{$k$-positioned  $\CaC$-semigroups}\label{Sec2}

A $\CaC$-semigroup $S$ is $k$-positioned for some $k\in \CaC$ if for every $h\in \CaH(S)$ we have that $k-h$ belongs to $S$. Notice that every $\CaC$-semigroup is $k$-positioned for a large enough $k\in \CaC$. 

This section discusses the concept of being $k$-positioned, providing several characterizations and key results. Our approach is inspired by the concept of being $k$-positioned, as discussed in \cite{Branco-Faria-Rosales} for numerical semigroups.

For a $\CaC$-semigroup $S$, consider the set 
\[X_S=\{x\in \CaC\mid h \leq_\CaC x\text{ for all }h\in \CaH(S)\}.\]
Observe that, if $S$ is a numerical semigroup, then $\mathrm{Minimals}_{\leq_\CaC}(X_S)=\{\operatorname{F}(S)\}$.   If $S$ is a GNS in $\mathbb{N}^d$, we have $\mathrm{Minimals}_{\leq_\CaC}(X_S)=\{\operatorname{Cr}(S)\}$, where each $i$-th coordinate of $\operatorname{Cr}(S)$ is defined as $\operatorname{Cr}(S)^{(i)}=\max\{h^{(i)} \mid h\in \CaH(S)\}$. 
Consider $\{e_1,\ldots, e_d\}$ the set of canonical basis vectors of $\mathbb{N}^d$, the invariant $\operatorname{Cr}(S) +\sum_{i=1}^d e_i$ is called the \textit{corner element} of $S$, first introduced in \cite{berbardini-et-all} (see also \cite{bernardini2024atoms} for some other properties related to it). If $S$ is not a GNS, then the cardinality of the set
$\mathrm{Minimals}_{\leq_\CaC}(X_S)$ can be greater than one, as shown in the following example.

\begin{example}
Consider the positive integer cone $\CaC$ spanned by $\{(4,1), (5,3)\}$, and the $\CaC$-semigroup $S$ minimally generated by 
\[\{(4,1), (4,2),(5,2),(5,3), (7,3), (7,4), (10,3), (11,3)\},\]
whose set of gaps is $\CaH(S)=\{(2,1), (3,1), (6,2), (6,3), (7,2)\}$. We obtain that $\mathrm{Minimals}_{\leq_\CaC}(X_S)=\{(12,5), (13,5), (14,5)\}$.
\end{example} 

From now on, suppose that $S$ is a $k$-positioned $\CaC$-semigroup for some $k\in \CaC$. It is not difficult to see that $k\in X_S$. In the context of numerical semigroups, it is known that $k=\operatorname{F}(S)$ if and only if $S$ is symmetric (see \cite{Branco-Faria-Rosales}). We generalize this property in a more general framework. 

\begin{proposition}\label{prop:k-gap}
Let $S$ be a $\CaC$-semigroup. Then $S$ is $k$-positioned for some $k\in \CaH(S)$ if and only if $S$ is symmetric and $\operatorname{F}(S)=k$.
\end{proposition}

\begin{proof}
Suppose that $S$ is $k$-positioned and $k\in \CaH(S)$. Consider the map
$$ \varphi_S: \CaH(S) \longrightarrow \operatorname{I}_S(k) \quad \text{ defined by } \quad h \longmapsto k-h.$$
Since $S$ is $k$-positioned, the map $\varphi$ is well-defined, and by construction, it follows that it is injective. Therefore $\operatorname{g}(S)\leq |\operatorname{I}_S(k)|$. Let $s\in  \operatorname{I}_S(k)$, it follows that $h=k-s$ belongs to $\CaH(S)$, otherwise $k\notin \CaH(S)$. In particular, $\varphi(h)=s$, and thus the map is bijective. So, $\operatorname{g}(S)=|\operatorname{I}_S(k)|$. Since $k-h\in S$ for all $h\in \CaH(S)$, we deduce that $h\leq_\CaC k$ for all $h\in \CaH(S)$ and using Corollary \ref{unique-Fb}  this implies that $k=\operatorname{F}(S)$ and by applying Proposition \ref{caracSymm} we conclude that $S$ is symmetric.

\noindent Conversely, since $k=\operatorname{F}(S)$, and by Proposition \ref{caracSymm} we obtain that $\operatorname{F}(S)-h\in S$ for all $h\in \CaH(S)$,
which is the definition of $k$-positioned.
\end{proof}



Suppose $S$ is a $k$-positioned $\CaC$-semigroup with $k\in \CaH(S)$. So, for every $x\in X_S$, we have $h\leq_\CaC x$ for all $h\in \CaH(S)$, in particular $k\leq_\CaC x$, and this fact implies that $\mathrm{Minimals}_{\leq_\CaC}(X_S)=\{k\}$. Assuming now that $S$ is $k$-positioned with $\mathrm{Minimals}_{\leq_\CaC}(X_S)=\{k\}$. It is natural to ask whether $k\in \CaH(S)$. While this holds for numerical semigroups, as we have explained above, but it is not generally the case for $\CaC$-semigroups, as we illustrate in the following example. 

\begin{example}
Let $S$ be the GNS in $\mathbb{N}^2$, defined in Figure \ref{FIG:positionedNOTsim}, which set of gaps is $$\CaH(S)=\{(0,1), (0,3), (1,0), (1,1), (1,2), (1,3), (2,0), (2,1), (4,1)\}.$$
So,  $\operatorname{Cr}(S)=\{(4,3)\}$ and it is easy to check that  $S$ is  $(4,3)$-positioned, but $S$ is not symmetric since $\operatorname{PF}(S)=\{ (0,3), (1,2), (1,3), (2,0), (2,1), (4,1)\}$.
\begin{figure}[h]
    \centering
    \includegraphics[width=0.45\linewidth]{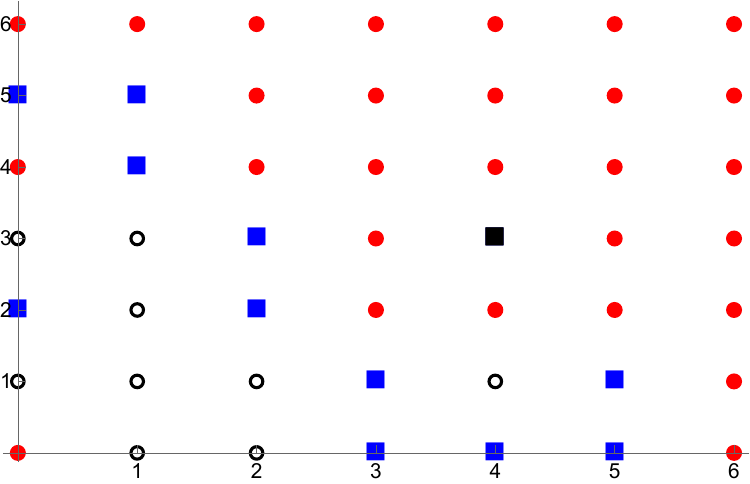}
    \captionsetup{width=0.62\textwidth}
    \caption{$\mathbb{N}^2$-semigroup $S$, $\circ\in\CaH(S)$; \textcolor{black}{\scalebox{0.6}{$\blacksquare$}} $\equiv\operatorname{Cr}(S)$; \textcolor{red}{$\bullet$} $\in S$; \textcolor{blue}{\scalebox{0.6}{$\blacksquare$}} $\in \operatorname{msg}(S)$}
    \label{FIG:positionedNOTsim}
\end{figure}
\end{example}

Based on the previous considerations, given a $k$-positioned $\CaC$-semigroup, we assume that $k$ is an element of the $\CaC$-semigroup.

If $S$ is a $k$-positioned numerical semigroup, it is known that $k\geq \operatorname{F}(S)+\operatorname{m}(S)$. In fact, if $\operatorname{F}(S)<k<\operatorname{F}(S)+\operatorname{m}(S)$, then $k$ can be expressed as $k=\operatorname{F}(S)+i$ with $i\in \{1,2,\ldots,\operatorname{m}(S)-1\}\subset \CaH(S)$ and $k-i\notin S$, whence there are no $k$-positioned numerical semigroups. This argument can be interpreted to $\CaC$-semigroups differently to obtain a generalization of the above inequality. 
To this purpose, for any $\CaC$-semigroup $S$, we introduce the following sets (inspired by similar definitions provided in \cite{cisto-generalization}).
\begin{itemize}
    \item[] $\operatorname{C}(S)= \{ x \in \CaC \mid x\leq_\CaC h, \text{ for some } h\in \CaH(S)\}$,
    \item[] $\operatorname{M}(S) = \big\{ h \in \CaH(S) \mid \operatorname{I}_S(h)=\{0\}\big\}\cup\{0\}$.
\end{itemize}

We present the following characterization.

\begin{proposition}\label{prop:M(S)=0}
    Let $S$ be a $\CaC$-semigroup. Then, $\operatorname{M}(S)=\{0\}$ if and only if $S=\CaC$.
\end{proposition}
\begin{proof}
    If $S=\CaC$ trivially $\operatorname{M}(S)=\{0\}$. Conversely, assume $\operatorname{M}(S)=\{0\}$ and suppose that there exists $h\in \CaH(S)$. So, there exist $z_1\in S\setminus \{0\}$ and $c_1\in \CaC$ such that $z_1+c_1=h$. In particular, since $h\notin  S$, we deduce $c_1\in \CaH(S)$. Again, we have $c_1\notin \operatorname{M}(S)$ and applying the same argument we obtain $z_2\in S\setminus \{0\}$ and $c_2\in \CaH(S)$ such that $z_2+c_2=c_1$. Continuing this process, we obtain an infinite sequence of elements $\{c_i\}_{i\geq 1}\subseteq \CaH(S)$, which is impossible since $S$ is a $\CaC$-semigroup. Hence, $S=\CaC$.
\end{proof}

Observe that if $S$ is a numerical semigroup, then $|\operatorname{M}(S)|=\operatorname{m}(S)$  and $|\operatorname{C}(S)|=\operatorname{F}(S)+1$.  

\begin{proposition}\label{prop:map-MC}
Let $S$ be a $k$-positioned $\CaC$-semigroup with $k\in S$. Then, the following map
$$ \phi_S: \operatorname{M}(S) \longrightarrow \operatorname{I}_\CaC(k)\setminus \operatorname{C}(S) \quad \text{ defined by } \quad h \longmapsto k-h$$
is well-defined and injective. In particular, $|\operatorname{M}(S)|+|\operatorname{C}(S)|\leq |\operatorname{I}_\CaC(k)|$.
\end{proposition}

\begin{proof}
Since $S$ is $k$-positioned, if $h\in \operatorname{M}(S)$, then $k-h\notin  \operatorname{C}(S)$. Otherwise, there exists at least one element $y\in \CaH(S)$ such that $k-h\leq_\CaC y$. By using that $\leq_\CaC$ is compatible with the addition, $k-y\leq_\CaC h$. If $k-y\in \CaH(S)$ then $S$ is not $k$-positioned, therefore $k-y\in S$. Taking into account that $h\in \operatorname{M}(S)$, we have  $k-y=0$ and thus $k=y\in \CaH(S)$, which is a contradiction. By construction $\phi_S$ is injective, and thus $|\operatorname{M}(S)|+|\operatorname{C}(S)|\leq |\operatorname{I}_\CaC(k)|$.
\end{proof}

In line with Proposition \ref{prop:map-MC}, we say that a $\CaC$-semigroup is \textit{primary positioned} if it is $k$-positioned for some $k\in \CaC$, and $|\operatorname{M}(S)| + |\operatorname{C}(S)| = |\operatorname{I}_\CaC(k)|$. 
The property for a $\CaC$-semigroup $S$ of be primary positioned  for some $k$ implies that $k\in S$. Otherwise $\operatorname{C}(S)=\operatorname{I}_\CaC(k)$, and since $|\operatorname{M}(S)|\geq 1$ we have $|\operatorname{M}(S)| + |\operatorname{C}(S)| > |\operatorname{I}_\CaC(k)|$.

In the context of numerical semigroups, that is, if we assume that $S$ is a primary positioned numerical semigroup, the equality $|\operatorname{M}(S)|+|\operatorname{C}(S)|=|\operatorname{I}_\CaC(k)|$ specializes to $\operatorname{m}(S)+\operatorname{F}(S)+1= k+1$. So, it is worth mentioning that the primary positioned numerical semigroups correspond with the \textit{positioned} numerical semigroups introduced in \cite{Branco-Faria-Rosales}.

The following lemma presents a characterization relating the maximal gaps of $S$ and the minimal non-zero elements of $S$, assuming that $S$ is a $k$-positioned $\CaC$-semigroup.

\begin{lemma}\label{lemma:min-max}
Let $S$ be a primary positioned $\CaC$-semigroup for $k\in S$, and let $s\in \operatorname{I}_S(k)$. Then,  $k-s\in \operatorname{Maximals}_{\leq_\CaC}(\CaH(S))$ if and only if $s\in \operatorname{Minimals}_{\leq_\CaC}(S\setminus\{0\})$.
\end{lemma}
\begin{proof}
Suppose $k-s\in \operatorname{Maximals}_{\leq_\CaC}(\CaH(S))$. If $s\notin \operatorname{Minimals}_{\leq_\CaC}(S\setminus \{0\})$, then there exists $x\in S\setminus \{0\}$ such that $x<_\CaC s$. So, $k-s <_\CaC k-x$. By the maximality of $k-s$, we deduce $k-x\notin \operatorname{C}(S)$. Consider the injective map $\phi_S$ in Proposition~\ref{prop:map-MC}, since $S$ is primary positioned $\phi_S$ is bijective, then $x=k-(k-x)\in \operatorname{M}(S)\subseteq \CaH(S)$, a contradiction.  Now, suppose $s\in \operatorname{Minimals}_{\leq_\CaC}(S\setminus\{0\})$ such that $s\leq_\CaC k$, let us see that $k-s\in \CaH(S)$. Trivially, if $k-s\notin\operatorname{C}(S)$, by the property of being primary positioned, the map $\phi_S$ given in Proposition \ref{prop:map-MC} is bijective, and thus we deduce that $k-(k-s)=s\in \operatorname{M}(S)\subseteq \CaH(S)$, which is impossible. So, we assume that $k-s\in\operatorname{C}(S)$. By definition, take $h\in\operatorname{Maximals}_{\leq_\CaC}(\CaH(S))$ such that $k-s\leq_\CaC h$. In particular, $k-h\leq_\CaC s$ and $k-h\in S\setminus \{0\}$. Therefore, by the minimality of $s$, we have $k-h=s$. So, $k-s=h\in \CaH(S)$.
\end{proof}
The result below provides an upper bound on the genus of a $k$-positioned $\CaC$-semigroup.

\begin{proposition}\label{prop:upbound}
Let $S$ be a $k$-positioned $\CaC$-semigroup for some $k\in S$. Then, 
\begin{equation}\label{boundGenus}
    \operatorname{g}(S)\leq\dfrac{|\operatorname{I}_\CaC(k)|-2}{2}.
\end{equation}
\end{proposition}

\begin{proof}
Since $S$ be a $k$-positioned the map $\psi_S: \CaH(S) \longrightarrow \operatorname{I}_S(k)\setminus\{0,k\}$ with $\psi_S(h)=k-h$ is well-defined and injective. Therefore, $\operatorname{g}(S)\leq |\operatorname{I}_S(k)|-2$ and since $\operatorname{I}_\CaC(k)=\CaH(S)\sqcup \operatorname{I}_S(k)$, we have $|\operatorname{I}_\CaC(k)|=\operatorname{g}(S)+|\operatorname{I}_S(k)|\geq 2\operatorname{g}(S)+2$. So, we obtain the desired inequality.
\end{proof}

As an immediate consequence of the above proposition, we obtain the following result.

\begin{corollary}
Let $S$ be a $k$-positioned GNS in $\mathbb{N}^d$ for some $k\in S$. Then, 
$$\operatorname{g}(S)\leq\dfrac{\prod_{i=1}^d(k^{(i)}+1)-2}{2}.$$
\end{corollary}

\section{UESY-semigroups and PEPSY-semigroups}\label{Sec3}

We say that any $\CaC$-semigroup $S$ is a \textit{UESY- semigroup} (unitary extension of a symmetric $\CaC$-semigroup) if there exists a symmetric $\CaC$-semigroup $T$ such that $S=T\cup\{\operatorname{F}(T)\}$. Similarly, we say that $S$ is a \textit{PEPSY-semigroup} if there exists a pseudo-symmetric $\CaC$-semigroup $T$ such that $S=T\cup\{\frac{\operatorname{F}(T)}{2},\operatorname{F}(T)\}$. These definitions are the natural extension of the concept of UESY-semigroups and PEPSY-semigroups given in \cite{UESY} and \cite{PEPSY}, respectively, for numerical semigroups. 

The main objective of this section is to study the connection between the two previous classes and the property of being primary positioned.

The following proposition, which extends \cite[Theorem 1.8]{UESY},  establishes the equivalence between UESY-semigroups and $k$-positioned $\CaC$-semigroups with the maximum possible number of gaps inside $\operatorname{I}_\CaC(k)$. 

\begin{proposition}\label{Prop:symm&k-positioned}
Let $S$ be a $\CaC$-semigroup and let $k\in S$. The following properties are equivalent:
\begin{enumerate}
    \item[1)] $S$ is $k$-positioned and $\operatorname{g}(S)=\dfrac{|\operatorname{I}_\CaC(k)|-2}{2}$.
    \item[2)] $S$ is $k$-positioned and $k\in \operatorname{msg}(S)$.
    \item[3)] $S$ is a UESY-semigroup.
\end{enumerate}
\end{proposition}

\begin{proof}
$1)$ implies $2)$. Consider the injective map $\psi_S$ given in the proof of Proposition \ref{prop:upbound}. By hypothesis, it follows that $\psi_S$ is bijective. Suppose that $k\notin \operatorname{msg}(S)$ so $k=s_1+s_2$ for some $s_1,s_2\in S\setminus\{0\}$. In particular, $s_1\in \operatorname{I}_S(k)\setminus\{0,k\}$ and therefore, there exists $h\in \CaH(S)$ such that $\psi_S(h)=s_1$. So, $k-h=s_1=k-s_2$ which implies that $h=s_2\in S$, a contradiction.

\noindent $2)$ implies $3)$. Since $k\in \operatorname{msg}(S)$, then $S\setminus\{k\}=T$ is a $\CaC$-semigroup, so $k\in \CaH(T)$. To prove that $S$ is UESY-semigroup, let us show that $T$ is symmetric, and according to the result given in Proposition \ref{prop:k-gap}, it suffices to check that $T$ is $k$-positioned. Let $h\in \CaH(T)$. If $h=k$ then $k-h=0\in T$. Otherwise $h\in \CaH(S)$ and since $S$ is $k$-positioned it follows that $k-h\in S\subset T$, which completes the proof. 

\noindent $3)$ implies $1)$. If $S$ is a UESY-semigroup, then $S=T\cup\{\operatorname{F}(T)\}$, where $T$ is symmetric. Take $k=\operatorname{F}(T)$. By applying Proposition~\ref{caracSymm}, it follows that $\operatorname{g}(T)=\frac{|\operatorname{I}_\CaC(k)|}{2}$. Observe that $\operatorname{g}(S)=\operatorname{g}(T)-1$, so $\operatorname{g}(S)=\frac{|\operatorname{I}_\CaC(k)|-2}{2}$. Now, let $h\in \CaH(S)=\CaH(T)\cup \{k\}$. By the symmetry of $T$ we have $k-h\in T\setminus \{k\}=S$. Thus, $S$ is $k$-positioned.
\end{proof}

From the previous proposition, we state the next result.

\begin{theorem}\label{Theo:UESY}
If $S$ is an UESY-semigroup, then is primary positioned.
\end{theorem}

\begin{proof}
By definition of be UESY-semigroup, $S=T\cup\{\operatorname{F}(T)\}$ for some symmetric $\CaC$-semigroup $T$. By applying Proposition \ref{Prop:symm&k-positioned}, $S$ is $k$-positioned with $k=\operatorname{F}(T)$. Consider the injective map $\varphi_S$ given in Proposition \ref{prop:map-MC}, to show that $S$ is primary positioned, it suffices to prove that $\varphi_S$ is bijective. Trivially, if $x=k$ then $k-x=0\in\operatorname{M}(S)$. So, let $x\in \operatorname{I}_\CaC(k)\setminus (\operatorname{C}(S)\cup\{k\})$. Observe that if $k-x\in S\setminus\{k\}=T$, by the symmetry of $T$ it follows that $x\in \CaH(T)=\CaH(S)\cup\{k\}$, and thus $x\in \CaH(S)\subset \operatorname{C}(S)$, which it is not possible. Therefore, $k-x\in \CaH(S)$. Suppose now that $k-x\notin \operatorname{M}(S)$, that is, there exists $z\in S\setminus\{0\}$ such that $z\leq_\CaC k-x$, which implies that $x\leq_\CaC k-z$. 
Since $x\neq 0$, we have $z\neq k$. In particular, $z\in T$ and thus $k-z\in \CaH(T)$, otherwise, if $k-z$ belongs to $T$, we have $k=(k-z)+z\in T$, which contradicts $k=\operatorname{F}(T)$. So, $k-z\in \CaH(S)\setminus\{k\}\subset \operatorname{C}(S)$ and consequently, we obtain $x\in\operatorname{C}(S)$, which is false. We conclude that $k-x\in \operatorname{M}(S)$ for every $x\in \operatorname{I}_\CaC(k)\setminus \operatorname{C}(S)$, which means that $\varphi_S$ is bijective.
\end{proof}

Returning to the case where $S$ is a numerical semigroup, we know that there exists a unique natural number $k = \operatorname{F}(S) + \operatorname{m}(S)$ such that $S$ is primary positioned, as previously shown. Motivated by this fact, we may inquire whether there exist two distinct values $k_1, k_2 \in \CaC$ for which $S$ is primary positioned. Assuming  that $S$ is an UESY-semigroup, the answer is given in the following example.

\begin{example}\label{ex:2UESY}
Let $\CaC$ be the cone spanned by $\{(1,0),(1,1)\}$ and let $S=\CaC \setminus \{(1,0),(1,1),(2,0),(2,1),(2,2),(3,1),(3,2),(4,2)\}$ be the $\CaC$-semigroup given in Figure \ref{fig:Countraex.UESY}. Notice that $T_1=S\setminus\{(7,5)\}$ and $T_2=S\setminus\{(7,2)\}$ are two symmetric $\CaC$-semigroups. So, $S$ is a UESY-semigroup for $T_1, T_2$. By Theorem \ref{Theo:UESY}, we conclude that $S$ is primary positioned for two different $k_1=(7,5)$ and $k_2=(7,2)$.
\begin{figure}[h]
    \centering
    \includegraphics[width=0.5\linewidth]{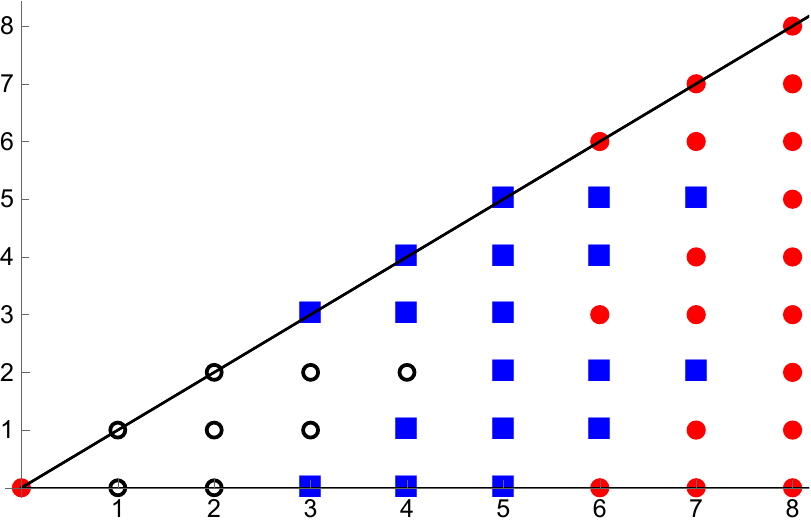}
    \captionsetup{width=0.56\textwidth}
    \caption{Primary positioned $\CaC$-semigroup $S$, $\circ\in\CaH(S)$;  \textcolor{red}{$\bullet$} $\in S$; \textcolor{blue}{\scalebox{0.6}{$\blacksquare$}} $\in \operatorname{msg}(S)$}
    \label{fig:Countraex.UESY}
\end{figure}
\end{example}

Next, we focus on obtaining similar results to those presented in Proposition \ref{Prop:symm&k-positioned} and Theorem \ref{Theo:UESY}, for the case of PEPSY-semigroups. The following result is inspired by \cite[Theorem 33]{PEPSY}, which is formulated for numerical semigroups, and it extends to $\CaC$- semigroups. We say that an element $x\in S$ has a unique expression in $S$ if there does not exist two different pairs $(s_1,s_2),(s_3,s_4)\in S^*\times S^*$  such that $x=s_1+s_2=s_3+s_4$.

\begin{proposition}\label{Prop:PEPSY&k-positioned}
Let $S$ be a $\CaC$-semigroup and let $k\in S$. The following properties are equivalent:
\begin{enumerate}
    \item[1)] $S$ is $k$-positioned, and $\operatorname{g}(S)=\dfrac{|\operatorname{I}_\CaC(k)|-3}{2}$.
    \item[2)] $S$ is $k$-positioned, and $k$ has a unique expression in $S$.
    \item[3)]  $S$ is a PEPSY-semigroup.
\end{enumerate}
\end{proposition}
\begin{proof}
$1)$ implies $2)$. From the definition of $\operatorname{I}_\CaC(k)$ we have that $\operatorname{g}(S)=|\operatorname{I}_S(k)|-3$. By the inequality~\eqref{boundGenus} and by the arguments contained in the proof of Proposition~\ref{prop:upbound}, we deduce that the injective map $\psi_S$ is not bijective, and there exists a unique element $x\in \operatorname{I}_S(k)\setminus\{k,0\}$ such that $k-x\notin \CaH(S)$. If $x\ne\frac{k}{2}$ then $k-x\in \operatorname{I}_S(k)\setminus\{0,x,k\} $ and this implies that there exists $h\in \CaH(S)$ such that $\psi_S(h)=k-x$. Therefore, $x=h\in \CaH(S)$ which is a contradiction. Hence, $x=\frac{k}{2}$ and $\psi'_{S}:\CaH(S)\longrightarrow \operatorname{I}_S(k)\setminus\{0,\frac{k}{2},k\}$  defined via $\psi'_{S}(h)=k-h$ is bijective. We prove that $k$ has a unique expression in $S$. Observe that, $k=\frac{k}{2}+\frac{k}{2}$ and we assume $k=s_1+s_2$ with $s_1,s_2\in S\setminus\{0,\frac{k}{2}\}$. Since $s_1\in \operatorname{I}_S(k)\setminus\{0,\frac{k}{2},k\}$ there exists $h\in \CaH(S)$ such that $\psi'_{S}(h)=s_1$. So, $k-s_2=s_1=k-h$, and this implies that $h=s_2\in S$, which is false.

\noindent $2)$ implies $3)$. 
Since $S$ is $k$-positioned necessary $\frac{k}{2}\in S$. Otherwise, i.e. $\frac{k}{2}\in\CaH(S)$, it follows that $k-\frac{k}{2}=\frac{k}{2}\in S$, which is false. Notice $k=\frac{k}{2}+\frac{k}{2}$ and by hypothesis is the unique expression in $S$. Let us show that this fact implies that $\frac{k}{2}\in \operatorname{msg}(S)$. Suppose that  
$\frac{k}{2}=s_1+s_2$ with $s_1,s_2\in S\setminus\{0,\frac{k}{2},k\}$. Therefore, $k=\frac{k}{2}+\frac{k}{2}=s_1+(s_2+s_1+s_2)$, contradicting the unique expression of $k$ in $S$. So, $\frac{k}{2}\in \operatorname{msg}(S)$ and thus $S\setminus\{\frac{k}{2}, k\}=T$ is a $\CaC$-semigroup. From the definition of be $k$-positioned, all the gaps of $S$ are smaller than $k$ with respect to $\leq_\CaC$, so we deduce that $\operatorname{F}(T)=k$. Let $h\in \CaH(T)\setminus\{\frac{k}{2}\}$, it follows that $h\in \CaH(S)$ and since $S$ is $\operatorname{F}(T)$-positioned then $\operatorname{F}(T)-h\in S\subset T$. Hence, $T$ is pseudo-symmetric and we conclude that $S=T\cup\{\operatorname{F}(T), \frac{\operatorname{F}(T)}{2}\}$ is a PEPSY-semigroup.

\noindent $3)$ implies $1)$.  Let $S=T\cup\{\operatorname{F}(T), \frac{\operatorname{F}(T)}{2}\}$, where $T$ is a pseudo-symmetric $\CaC$-semigroup and fix $k=\operatorname{F}(T)$. By using Proposition~\ref{caracPesudo}, we obtain that $\operatorname{g}(T)=\frac{|\operatorname{I}_\CaC(k)|+1}{2}$. Since $\operatorname{g}(S)=\operatorname{g}(T)-2$, then $\operatorname{g}(S)=\frac{|\operatorname{I}_\CaC(k)|-3}{2}$. Considering that $T$ is pseudo-symmetric, we have $k-h\in T\subset S $ for all $h\in \CaH(S)\subset \CaH(T)\setminus\{\frac{k}{2}\}$, which proves that $S$ is $k$-positioned.
\end{proof}

Based on the preceding proposition, we present this result.

\begin{theorem}\label{theo:PEPSY}
If $S=T\cup\{\operatorname{F}(T),\frac{\operatorname{F}(T)}{2}\}$ is a PEPSY-semigroup and $\frac{\operatorname{F}(T)}{2}\in \operatorname{C}(S)$, then $S$ is primary positioned.
\end{theorem}

\begin{proof}
Let $S=T \cup \{k,\frac{k}{2}\}$ where $T$ is pseudo-symmetric. By the proof of Proposition \ref{Prop:PEPSY&k-positioned} we have $S$ is $k$-positioned for $k=\operatorname{F}(T)$. Arguing as in Theorem \ref{Theo:UESY}, consider the injective map $\varphi_S$ given in Proposition \ref{prop:map-MC} and let us see that $\varphi_S$ is bijective. Trivially, if $x=k$ then $k-x=0\in\operatorname{M}(S)$. So, let $x\in \operatorname{I}_\CaC(k)\setminus (\operatorname{C}(S)\cup\{k\})$. We have to prove that $k-x\in \operatorname{M(S)}$. First, we show that $k-x\in \CaH(S)$. Assume $k-x\in S$, and we can distinguish three cases. If $k-x=k$, then $x=0$ which implies $x\in \operatorname{C}(S)$, a contradiction. If $k-x=\frac{k}{2}$, then $x=\frac{k}{2}$ and by hypothesis, we deduce that  $x\in \operatorname{C}(S)$, which is not possible. If $k-x\in T$, then $x$ belongs to $\CaH(T)\setminus\{k,\frac{k}{2}\}$, since $T$ is pseudo-symmetric. So, $x\in\CaH(S)\subset\operatorname{C}(S)$, obtaining again a contradiction. Therefore, $k-x\in \CaH(S)$. So, if we suppose $k-x\notin \operatorname{M}(S)$ it means that, there exists $z\in S\setminus\{0\}$ such that $z\leq_\CaC k-x$. In particular $x\leq_\CaC k-z$. Observe that if $z\in\{k,\frac{k}{2}\}$, then $x\in \operatorname{C}(S)$. So, $z\in T\setminus\{0\}$ and by the pseudo-symmetry property of $T$, we have $k-z\in \CaH(T)\setminus \{k,\frac{k}{2}\}=\CaH(S)$. Consequently, we obtain $x\in \operatorname{C}(S)$ which is false. This completes the proof.
\end{proof}

We notice in Example~\ref{ex:2UESY} that it is possible to find an UESY-semigroup that is primary positioned for two different elements $k_1,k_2\in \CaC$. A natural question that arises is whether we have the same result for some PEPSY- semigroups. That is, we ask if there exists a PEPSY-semigroup that is primary positioned for two different elements $k_1,k_2\in \CaC$. As far as the authors are aware, no examples are known. We leave the question as an open problem.

\section{Existence of primary positioned $\CaC$-semigroups}\label{Sec4}

This section aims to establish the existence of a GNS primary positioned  for a given $k \in \CaC$ such that $|\operatorname{I}_\CaC(k)| \notin \{3, 5,7\}$, while also showing that given some $k\in\CaC$ there does not exist a non-GNS primary positioned. To achieve these goals, we begin by analyzing the parity of $\operatorname{I}_\CaC(k)$.

\begin{proposition}\label{prop:even-pp}
Let $\CaC$ be a positive integer cone and let $k\in\CaC$. If $|\operatorname{I}_\CaC(k)|$ is even, then there exists a primary positioned $\CaC$-semigroup for $k$.
\end{proposition}
\begin{proof}
Suppose $|\operatorname{I}_{\CaC}(k)|$ is even. By \cite[Algorithm 1]{garcia-on-some} it is possible to produce an irreducible $\CaC$-semigroup $T$ such that $\operatorname{F}(T)=k$. Furthermore, applying Proposition \ref{caracSymm} we have that $T$ is symmetric. Consider the UESY-semigroup $S=T\cup\{\operatorname{F}(T)\}$, using Theorem \ref{Theo:UESY} it follows that $S$ is primary positioned for $k$.
\end{proof}

Before addressing the case when $ |\operatorname{I}_\CaC(k)|$ is an odd number, we first present the following remark regarding the case when $|\operatorname{I}_{\CaC}(k)| = 1$, and a technical lemma.

\begin{remark}\label{rem:k=0}
 If $|\operatorname{I}_{\CaC}(k)|=1$ then $k=0\in\CaC$. Consider $S=\CaC$. Since $|\operatorname{M}(S)|=1$ and $|\operatorname{C}(S)|=0$, we conclude that $S$ is primary positioned for $k=0$.
\end{remark}

\begin{lemma}\label{lem:k/2}
    Let $\CaC$ be an integer cone and let $k\in \CaC$. If 
    $|\operatorname{I}_{\CaC}(k)|$ is odd, then $\frac{k}{2}\in \CaC$.
\end{lemma}
\begin{proof}
Trivially, if $|\operatorname{I}_{\CaC}(k)|=1$ then $k=0\in \CaC$ and thus $\frac{k}{2}=0\in \CaC$. Assume that $|\operatorname{I}_{\CaC}(k)|=r$ where $r\geq 3$ is an odd number and fix a term order $\preceq$. So, 
    \[
    \operatorname{I}_{\CaC}(k)=\big\{0=x_1\prec\ldots \prec x_{\frac{r-1}{2}-1}\prec x_{\frac{r-1}{2}}\prec x_{\frac{r-1}{2}+1} \prec \ldots \prec x_r=k\big\}.
    \]
    By definition,  $x\in \operatorname{I}_{\CaC}(k)$ if and only if $x+z=k$, for some $z\in \CaC$, equivalently $z\in \operatorname{I}_{\CaC}(k)$. Since we are under the hypothesis of 
    \cite[Lemma 3.1]{cisto-tenorio}, we obtain that,
\begin{multline*}
    \operatorname{I}_{\CaC}(k)= \left \lbrace 0=x_1\prec x_2\prec\ldots 
    \prec x_{\frac{r-1}{2}-1}
    \prec x_{\frac{r-1}{2}}=k-x_{\frac{r-1}{2}}\prec \right.\\
    \left. x_{\frac{r-1}{2}+1}=k- x_{\frac{r-1}{2}-1} \prec \ldots 
    \prec x_{r-1}=k- x_1
    \prec x_r=k \right \rbrace.
\end{multline*}
In particular, $x_{\frac{r-1}{2}}=k-x_{\frac{r-1}{2}}$, so $x_{\frac{r-1}{2}}=\frac{k}2$. Hence, $\frac{k}{2}\in \CaC$.
\end{proof}

\begin{proposition}\label{prop:primary-odd}
There not exists any primary positioned $\CaC$-semigroup for $k\in \CaC$ such that $|\operatorname{I}_\CaC(k)|\in\{3,5,7\}$.    
\end{proposition}
\begin{proof}
Suppose that $S$ is primary positioned and $|\operatorname{I}_\CaC(k)|=n \in\{3,5,7\}$. In particular, $|\operatorname{M}(S)|+|\operatorname{C}(S)|=n$. By Lemma \ref{lem:k/2}, it follows that $\frac{k}{2}\in \CaC$. Observe  that $\frac{k}{2}\in S$, since $k-\frac{k}{2}=\frac{k}{2}$ and $S$ is $k$-positioned.   First, suppose $|\operatorname{M}(S)|=1$. By Proposition~\ref{prop:M(S)=0} we obtain $S=\mathcal{C}$, which implies $|\operatorname{C}(S)|=0$, a contradiction. So, we can assume  $|\operatorname{M}(S)|>1$. In this case $\mathcal{H}(S)\neq \emptyset$ and $\operatorname{M}(S)\subseteq \operatorname{C}(S)$, which implies $|\operatorname{M}(S)|\leq |\operatorname{C}(S)|$. Notice also that $\CaH(S)\subseteq\operatorname{I}_\CaC(k)\setminus\{0,\frac{k}{2}, k\}$. Hence, the case $n=3$ is not possible. 

Suppose $n=5$. Let $\preceq$ be a term order, by the proof of Lemma \ref{lem:k/2} it follows that $\operatorname{I}_\CaC(k)=\{0\prec h\prec\frac{k}{2}\prec k-h\prec k\}$, where $h\in \CaH(S)$ or $k-h\in \CaH(S)$. Additionally, note that it is not possible that $\{h,k-h\}=\CaH(S)$, since $S$ is $k$-positioned. Hence, the only possibility is $\operatorname{M}(S)=\{0,g\}$, with $g\in \{h,k-h\}$ and $\mathcal{H}(S)=\{g\}$. In particular, if there exists $x\in \operatorname{C}(S)$ such that $x\notin \operatorname{M}(S)$, then $x\leq_\mathcal{C} g$ and $x\in S\{0\}$, that contradicts $g\in \operatorname{M}(S)$. It follows $\operatorname{M}(S)=\operatorname{C}(S)$ and $|\operatorname{M}(S)|+|\operatorname{C}(S)|=2+2\neq 5$, a contradiction. Hence, also the case $n=5$ is not possible.

Suppose $n=7$. Let $\preceq$ be a term order, by the proof of Lemma \ref{lem:k/2} it follows that $\operatorname{I}_\CaC(k)=\{0\prec h_1 \prec h_2 \prec\frac{k}{2}\prec k-h_2 \prec k-h_1\prec k\}$, where $h_i\in \CaH(S)$ or $k-h_i\in \CaH(S)$, for $i\in \{1,2\}$. Additionally, since $S$ is $k$-positioned, it is not possible that $\{h_i,k-h_i\}\subseteq \CaH(S)$ for each $i\in \{1,2\}$. So, we deduce that $\operatorname{g}(S)\leq 2$ and thus $|\operatorname{M}(S)|\in \{2,3\}$. Suppose $\operatorname{M}(S)=\{0,g_1,g_2\}$, then $\mathcal{H}(S)=\{g_1,g_2\}$. Using a similar argument as before, if there exists $x\in \operatorname{C}(S)$ such that $x\notin \operatorname{M}(S)$, then $x\leq_\mathcal{C} g_i$, for some $i\in \{1,2\}$, and $x\in S\setminus\{0\}$. This fact contradicts $g_i\in \operatorname{M}(S)$. Therefore, the only possibility is $|\operatorname{M}(S)|=2$. Assume $\operatorname{M}(S)=\{0,g_1\}$. Observe that the case $\operatorname{g}(S)=1$ is not possible in this situation, otherwise $\operatorname{M}(S)=\operatorname{C}(S)=\{0,g_1\}$. So, we can state $\mathcal{H}(S)=\{g_1,g_2\}$. Since $g_2\notin \operatorname{M}(S)$ and $|\operatorname{C}(S)|=5$ then, there exists $s\in S\setminus \{0\}$ such that $s\leq_\mathcal{C} g_2$. In particular, $s\in \operatorname{C}(S)\subseteq \operatorname{I}_\mathcal{C}(k)$. So, $s\in \{\frac{k}{2},k-g_1,k-g_2\}$. If $s=k-g_1$, then $k-g_1\leq_\mathcal{C} g_2$, obtaining $k-g_2\leq_\mathcal{C} g_1$, a contradiction by $k-g_2\in S\setminus \{0\}$ and $g_1\in \operatorname{M}(S)$. If $s=\frac{k}{2}$, by $\frac{k}{2}\leq_\mathcal{C} g_2$ we obtain $k-g_2\leq_\mathcal{C} \frac{k}{2}\leq_\mathcal{C} g_2$, that is, $k-g_2\in \operatorname{C}(S)$. If $s=k-g_2$, by $k-g_2\leq_\mathcal{C} g_2$ we obtain $ \frac{k}{2} \leq_\mathcal{C} g_2$, that is, $\frac{k}{2}\in \operatorname{C}(S)$. So, since $|\operatorname{C}(S)|=5$, in every case, the only possibility is $\operatorname{C}(S)=\{0,g_1,g_2, \frac{k}{2},k-g_2\}$. Now, suppose $g_1\in \operatorname{SG}(S)$. Then $T=S\cup \{g_1\}=\mathcal{C}\setminus \{g_2\}$ is a $\mathcal{C}$-semigroup with $\operatorname{M}(T)=\{0\}$ and $T\neq \mathcal{C}$, a contradiction by Proposition~\ref{prop:M(S)=0}. Therefore, $g_1\notin \operatorname{SG}(S)$. It follows, $2g_1=g_2$ or $g_1+t=g_2$ for some $t\in \{k-g_2, \frac{k}{2}\}$. So, $\operatorname{I}_\mathcal{C}(g_2)=\{0,g_1, k-g_2,\frac{k}{2},g_2\}$. First, observe that if $g_1+k-g_2=g_2$, then $g_2-\frac{k}{2}=\frac{g_1}{2}$ and since $\frac{k}{2}\leq _\mathcal{C} g_2$, we obtain $\frac{g_1}{2}\in \mathcal{C}$. But this implies $0 \neq \frac{g_1}{2}\leq_ \mathcal{C} g_1$, which contradicts $g_1\in \operatorname{M}(S)$ and the fact that there is no gap smaller than $g_1$. So, we have to examine the following cases: 
\begin{itemize}
\item Suppose $g_2=2g_1$. In this case, by the proof of Lemma~\ref{lem:k/2}, we deduce $\operatorname{I}_\mathcal{C}(g_2)=\{0 \prec k-g_2 \prec g_1 \prec \frac{k}{2} \prec g_2\}$ or $\operatorname{I}_\mathcal{C}(g_2)=\{0 \prec \frac{k}{2} \prec g_1 \prec k-g_2 \prec g_2\}$. Hence, in any case we have $g_2-\frac{k}{2}=k-g_2$, that is, $g_2=\frac{3k}{4}$. As a consequence $k-g_2=\frac{k}{4}\in S$, obtaining $\frac{k}{4}+\frac{k}{4}+\frac{k}{4}=g_2\in S$, a contradiction.
\item Suppose $g_2=g_1+\frac{k}{2}$. In this case, by the proof of Lemma~\ref{lem:k/2}, we deduce $\operatorname{I}_\mathcal{C}(g_2)=\{0 \prec g_1 \prec k-g_2 \prec \frac{k}{2} \prec g_2\}$ or $\operatorname{I}_\mathcal{C}(g_2)=\{0 \prec \frac{k}{2} \prec k-g_2 \prec g_1 \prec g_2\}$. Hence, in any case we have $g_2=2(k-g_2)$, that is, $g_2=\frac{2k}{3}$. As a consequence $k-g_2=\frac{k}{3}\in S$, obtaining $\frac{k}{3}+\frac{k}{3}=g_2\in S$, a contradiction.
\end{itemize}
We obtained a contradiction in every case, so also the equality $|\operatorname{M}(S)|=2$ is not possible. This concludes the proof.
\end{proof}

As seen in the previous proposition, we have considered the case when $|\operatorname{I}_\CaC(k)|$ takes values in $\{3, 5,7\}$. However, for $|\operatorname{I}_\CaC(k)| > 7$ the situation is still unsolved. To discuss this case, we study the following class of $\CaC$-semigroups given in the example below.

\begin{example}
    Let $\lambda$ be an even positive integer and $\CaC_\lambda \subseteq \mathbb{N}^2$ be the positive integer cone spanned by the set $\{(1,0),(1,\lambda)\}$. Consider $k=(2,\lambda)$. So, 
    \[\operatorname{I}_{\CaC_\lambda}(k)=\{(1,i)\in \mathbb{N}^2\mid 0\leq i\leq \lambda\} \cup \{(0,0),(2,\lambda)\},\]
    and thus $|\operatorname{I}_{\CaC_\lambda}(k)|=\lambda+3$. Observe that for every $i,j\in \{0,\ldots,\lambda\}$ such that $i<j$ we have $(1,i)\nleq_{\CaC_\lambda} (1,j)$. 
    Suppose that $S$ is a primary positioned $\CaC_\lambda$-semigroup for $k$. By definition $\operatorname{M}(S)\subseteq \operatorname{C}(S)$. If $h\in \operatorname{C}(S)\setminus \operatorname{M}(S)$, then there exists $z\in S\setminus \{0\}$ such that $z<_{\CaC_\lambda} h$. By construction, $h,z\in \operatorname{I}_{\CaC_\lambda}(k)\setminus \{0,k\}$, which is impossible since as we have seen all the elements in the set $\operatorname{I}_{\CaC_\lambda}(k)\setminus \{0,k\}$ are not comparable each other with respect to the partial order $\leq_{\CaC_\lambda}$. As a consequence,  $\operatorname{C}(S)=\operatorname{M}(S)$, and thus $|\operatorname{I}_{\CaC_\lambda}(k)|$ is even, which contradicts the parity of the cardinality of the set. 
\end{example}

Given any positive integer cone $\CaC$ and $k \in \CaC$, as we have observed in the example, it is not possible to construct in general a $\CaC$-semigroup that is primary positioned for $k$. The key reason why this is not feasible lies in the properties of the set $\operatorname{I}_{\CaC_\lambda}(k)$. Specifically, for some $x, y \in \operatorname{I}_{\CaC_\lambda}(k) \setminus \{0, k\}$, it must hold that $x \leq_\CaC y$ or $y \leq_\CaC x$. Consequently, if we take the cone $\CaC = \mathbb{N}^d$, this property is satisfied, allowing us to state the following theorem.

\begin{theorem}\label{Theo:GNSprimarypositioned}
Let $k\in \mathbb{N}^d$. There exists a primary positioned GNS for $k$ if and only if $k\notin \{2e_i, 4e_i, 6e_i\mid i \in \{1,\ldots,d\}\}$.
\end{theorem}

\begin{proof}
\textit{Necessary:} If $k_1=2e_i$, $k_2=4e_i$ or $k_3=6e_i$ then $|\operatorname{I}_{\mathbb{N}^d}(k_1)|=3$, $|\operatorname{I}_{\mathbb{N}^d}(k_2)|=5$ and $|\operatorname{I}_{\mathbb{N}^d}(k_3)|=7$, respectively. By Proposition \ref{prop:primary-odd}, we conclude that no primary positioned GNS exists. 

\noindent\textit{Sufficiency:} Let $k\in \mathbb{N}^d\setminus\{2e_i, 4e_i, 6e_i \mid i \in \{1,\ldots,d\}\}$.  Observe that $|\operatorname{I}_{\mathbb{N}^d}(k)|\notin \{3,5,7\}$. We distinguish two cases, depending on the parity of $|\operatorname{I}_{\mathbb{N}^d}(k)|$. If $|\operatorname{I}_{\mathbb{N}^d}(k)|$ is even, we are under the hypothesis of Proposition \ref{prop:even-pp}, and we conclude that there exists a primary positioned GNS. Suppose $|\operatorname{I}_{\mathbb{N}^d}(k)|$ is an odd number different to $3,5$ and $7$. If $|\operatorname{I}_{\mathbb{N}^d}(k)|=1$, then $k=0\in\mathbb{N}^d$, and by Remark \ref{rem:k=0} we obtain that $S=\mathbb{N}^d$ is primary positioned for $k\in\CaC$. Now we examine the case $|\operatorname{I}_{\mathbb{N}^d}(k)|$ is an odd number greater than 7.

Let $\preceq$ be a term order and define $\mathcal{S}_k=\CaC\setminus\bigl(\{x\in \operatorname{I}_{\mathbb{N}^d}(k) \mid 0\ne x\preceq \frac{k}{2}\}\cup\{k\}\bigr)$. By Lemma \ref{lem:k/2} we have that $\frac{k}{2}\in \operatorname{I}_{\mathbb{N}^d}(k)$, and by \cite[Lemma 12]{garcia-on-some} we know that $\mathcal{S}_k$ is pseudo-symmetric, and thus $T=\mathcal{S}_k\cup\{k, \frac{k}{2}\}$ is a PEPSY-semigroup. Consider $x=\frac{k}{2}+e_{i_0}=\min_\preceq\{\frac{k}{2}+e_i\in \operatorname{I}_{\mathbb{N}^d}(k)\mid i\in \{1,\ldots,d\}\}$ and  $S=T\setminus\{x\}\cup\{k-x\}$. In particular, observe that $k-x=\frac{k}{2}-e_{i_0}\in \mathbb{N}^d$ and $e_{i_0}\lneq_{\mathbb{N}^d} \frac{k}{2}$, the equality cannot be reached since by hypothesis $k\ne 2e_{i_0}$. Therefore $e_{i_0}\in \CaH(T)$. Let us prove that $S$ is GNS. It suffices to show that $S$ is closed under addition, since the complement of $S$ in $\mathbb{N}^d$ is finite. Take $s_1,s_2\in S\setminus \{0\}$ and we distinguish the following cases: 
\begin{itemize}
    \item If $s_1,s_2\in T\setminus\{0,x\}$, then $s_1+s_2\in T$, since $T$ is a GNS. Suppose  $s_1+s_2=x=\frac{k}{2}+e_{i_0}$, therefore $s_1+s_2-e_{i_0}=\frac{k}{2}$ and without loss of generality, we can assume that $s_2-e_{i_0}\in \mathbb{N}^d$. It implies that $s_1\leq_{\mathbb{N}^d}\frac{k}{2}$. By definition of $T$ we have $s_1=\frac{k}{2}$, and thus we deduce that $s_2=e_{i_0}$ which is a contradiction, since $e_{i_0}\in \CaH(T)$. So, $s_1+s_2\in T\setminus\{x\}\subset S$.
    \item Suppose $s_1\in T\setminus\{0,x\}$, $s_2=k-x=\frac{k}{2}-e_{i_0}$. First, by construction, $T+y\subseteq T$ for every $y\in \mathbb{N}^d$. So, $s_1+(k-x)\in T$.  Assume $s_1+(k-x)\notin S$. The only possibility is $s_1+(k-x)=x$ and by simple operations we arrive to $s_1=2e_{i_0}\in T\setminus \{0\}$. So, by definition of $T$, $2e_{i_0}\succ \frac{k}{2}$ and equality is not possible since $k\neq 4e_{i_0}$. By the properties of the  term order, we obtain $\frac{k}{2}\lneq k \prec \frac{k}{2}+2e_{i_0}$. In terms of the $i_0$-coordinates, we deduce that $\frac{k^{(i_0)}}{2}\leq 2$. If the rest of the coordinates $j\neq i_0$ of $k$ are $0$,  then $k\in \{2e_{i_0},4e_{i_0}\}$, which contradicts our hypothesis. If there exists $j\neq i_0$ such that the $j$-coordinate is $k^{(j)}\neq 0$, we deduce that $k^{(j)}\geq2$ since all the coordinates of $k$ have to be even. So, $\frac{k^{(j)}}{2}\geq 1$, and thus
    $\frac{k}{2}+e_j\leq k$, that is $\frac{k}{2}+e_j\in \operatorname{I}_{\mathbb{N}^d}(k)$. On one hand, since $\frac{k^{(i_0)}}{2}\leq 2$, we obtain that $\frac{k}{2}+e_j+e_{i_0}\leq_{\mathbb{N}^d}k\prec \frac{k}{2}+2e_{i_0}$. On the other hand, by the minimality of $e_{i_0}$, we conclude that $\frac{k}{2}+2e_{i_0}\prec \frac{k}{2}+e_j+e_{i_0}$, since $\frac{k}{2}+e_{i_0}\prec \frac{k}{2}+e_j$, 
    contradicting the previous inequality. Hence, $s_1+s_2\in T\setminus\{x\}\subset S$.
    \item If $s_1=s_2=k-x=\frac{k}{2}-e_{i_0}$, then $s_1+s_2=k-2e_{i_0}$. Trivially, if $2(k-x)=x$ then $k=\frac{3}{2}x$. It is not difficult to argue that in this case we obtain $k=6e_{i_0}$, which is not possible by hypothesis. Moreover, $k-2e_{i_0}\succ\frac{k}{2}$, otherwise $k\preceq\frac{k}{2}+2e_{i_0}$, a contradiction as we have shown in the previous case. We conclude that $s_1+s_2\in S$.
\end{itemize}

Since $T$ is a PEPSY-semigroup, by Proposition~\ref{Prop:PEPSY&k-positioned}, we know that it is $k$-positioned and $\operatorname{g}(T)=\frac{|\operatorname{I}_\CaC(k)|-3}{2}$. So, by construction, we obtain that $S$ is $k$-positioned and $\operatorname{g}(S)=\operatorname{g}(T)=\frac{|\operatorname{I}_\CaC(k)|-3}{2}$. Using Proposition~\ref{Prop:PEPSY&k-positioned} again, we conclude $S$ is a PEPSY-semigroup. Finally, since $\frac{k}{2}\leq_{\mathbb{N}^d} x$ and $x\in \CaH(S)$, we have $\frac{k}{2}\in \operatorname{C}(S)$. By Theorem~\ref{theo:PEPSY}, we obtain $S$ is primary positioned for $k$.
\end{proof}
 
\section{An arrangement for the set of primary positioned $\CaC$-semigroup for a fixed $k\in \CaC$} \label{section-tree}

From now on, let $\CaC\subseteq \mathbb{N}^d$ be a positive integer cone, $\preceq$ a term order on  
$\mathbb{N}^d$, and let $k\in\CaC$. The set of all $\CaC$-semigroup that are primary positioned for $k\in \CaC$ is denoted by $\mathcal{P}(k)$. Hereafter,  $\mathcal{P}(k)$ is a non-empty set. The aim of this section is to provide a graphical classification of the set $\mathcal{P}(k)$. To achieve this, we begin by introducing the following notation. For any $\CaC$-semigroup $S$, we consider the set
\[
\operatorname{B}(S)=\left \lbrace x\in \operatorname{msg}(S)\mid x \in \operatorname{C}(S)\setminus \left(\operatorname{Minimals}_{\leq_\CaC}(S\setminus\{0\})\right), k-x\in S, x\neq \tfrac{k}{2}\right \rbrace, 
\]

Let us denote by $\mathcal{I}(k)$ the set of irreducible $\CaC$-semigroup having Frobenius element equal to $k$. We define \[
\operatorname{EI}(k)=\biggl \lbrace S\in \mathcal{P}(k)\mid S\in \Bigl \lbrace T\cup \{k\}, T\cup \left \lbrace \tfrac{k}{2} ,k\right \rbrace \Bigr \rbrace,\, T\in \mathcal{I}(k)\biggr \rbrace.
\]
Observe that if $S\in \operatorname{EI}(k)$, we have two possibilities depending on the parity of $|\operatorname{I}_\CaC(k)|$. If it is even, then $S$ is a UESY-semigroup. If it is odd, then $S$ is a PEPSY-semigroup. Now, let $T$ be an irreducible $\CaC$-semigroup having Frobenius element equal to $k$. If $T$ is symmetric, then $T\cup \{k\}$ is primary positioned by Theorem~\ref{Theo:UESY}. If $T$ is pseudo-symmetric, observe that $T\cup \{k\}$ is not $k$-positioned since $k=\frac{k}{2}+\frac{k}{2}$. Consider $S=T\cup \{\frac{k}{2},k\}$. 
From Theorem~\ref{theo:PEPSY} we know that if $\frac{k}{2}\in \operatorname{C}(S)$, then $S$ is primary positioned. Conversely, if
$\frac{k}{2}\notin \operatorname{C}(S)$, then $k-\frac{k}{2}=\frac{k}{2}\in \operatorname{M}(S)\subseteq \CaH(S)$, a contradiction. So, $S$ is primary positioned if and only if $\frac{k}{2}\in \operatorname{C}(S)$. From this discussion, we obtain the following:
\begin{remark}\label{rem:EI(k)}
The structure of \( \operatorname{EI}(k) \) depends on the parity of \( |\operatorname{I}_\CaC(k)| \):
\begin{itemize}
    \item If \( |\operatorname{I}_\CaC(k)| \) is even, then $
    \operatorname{EI}(k) = \{ T \cup \{k\} \mid T \in \mathcal{I}(k) \}.$
    \item If \( |\operatorname{I}_\CaC(k)| \) is odd,  then
    \[
    \operatorname{EI}(k) = \Big\{ T \cup \{\tfrac{k}{2}, k\} \mid T \in \mathcal{I}(k) \text{ and } \tfrac{k}{2} \in \operatorname{C}\left(T \cup \{\tfrac{k}{2}, k\}\right)\Big\}.
    \]
\end{itemize}
\end{remark}

The next result establishes a relationship between the sets $\operatorname{EI}(k)$ and $\operatorname{B}(T)$ for any $T\in \mathcal{P}(k)$.

\begin{proposition}\label{prop:empty}
    Let $T\in \mathcal{P}(k)$. Then, $\operatorname{B}(T)=\emptyset$ if and only if $T\in \operatorname{EI}(k)$.
\end{proposition}
\begin{proof}
     Suppose $T\in \operatorname{EI}(k)$. Therefore, $T$ is a UESY-semigroup or a PEPSY-semigroup such that $k$ is the Frobenius element of $T\setminus \{k\}$ or $T\setminus \{\frac{k}{2},k\}$. 
     If there exists $x\in \operatorname{B}(T)$, then $k=x+t$ with $t\in T\setminus \{0,\frac{k}{2},k\}$. In particular, if $T$ is a UESY-semigroup, we obtain that $k$ is not a minimal generator and if $T$ is a PEPSY-semigroup, we obtain that $k$, together with $k=\frac{k}{2}+\frac{k}{2}$, has more than one expression. These facts contradict Propositions~\ref{Prop:symm&k-positioned} and \ref{Prop:PEPSY&k-positioned}. So, $\operatorname{B}(T)=\emptyset$. 
     Conversely, let $T\in \mathcal{P}(k)$. Suppose that $T$ is neither a UESY-semigroup nor a PEPSY-semigroup. By Propositions \ref{Prop:symm&k-positioned} and \ref{Prop:PEPSY&k-positioned}, we know that $k$ is not a minimal generator and $k$ has more than one expression in $T$. 
     Since $k\notin \operatorname{msg}(S)$, $k=t_1+t_2$ with $t_1,t_2\in T\setminus\{0\}$. In case where $\frac{k}{2}\in \mathbb{N}^p$, we have $k=\frac{k}{2}+\frac{k}{2}$ is an expression of $k$ in $T$, since by the property of be $k$-positioned, we know that $\frac{k}{2}\in T$. Therefore, without loss of generality, we assume $t_1,t_2\notin \{0,\frac{k}{2},k\}$ and $t_1\in \operatorname{msg}(T)$. It follows that $t_1\in \operatorname{C}(T)$, otherwise, by the bijection of the map $\phi_T$ defined in Proposition \ref{prop:map-MC}, we obtain that  $k-t_1=t_2\in\CaH(T)$, a contradiction. By hypothesis $\operatorname{B}(T)=\emptyset$, so necessarily $t_1\in \operatorname{Minimals}_{\leq_\CaC}(T\setminus\{0\})$, and by applying Lemma \ref{lemma:min-max}, we have that $k-t_1=t_2\in\CaH(T)$, which is false. So, we conclude that $T$ is a UESY-semigroup or a PEPSY-semigorup.    
\end{proof}

Given a $\CaC$-semigroup $S$, the definition of $\operatorname{SG}(S)$ directly implies that its elements are precisely the gaps of $S$ such that $S \cup \{x\}$ remains a $\CaC$-semigroup. Accordingly, the following lemma shows that adding a special gap $x$ of $S$ to it does not affect the invariants $\operatorname{C}(S \cup \{x\})$, $\operatorname{Minimals}{\leq_\CaC}\bigl((S \cup \{x\}) \setminus \{0\}\bigr)$, and $\operatorname{M}(S \cup \{x\})$ under specific conditions.

\begin{lemma}\label{lemma:C-and-min-minus-x}
    Let $S$ be a $\CaC$-semigroup and let $x\in \operatorname{SG}(S)$ such that  $x \notin \operatorname{M}(S) \cup \operatorname{Maximals}_{\leq_\CaC}(\CaH(S))$. Then the following holds:
    \begin{enumerate}
        \item $\operatorname{C}(S \cup \{x\})=\operatorname{C}(S)$.
        \item $\operatorname{Minimals}_{\leq_\CaC}\bigl((S  \cup \{x\})\setminus \{0\}\bigr)=\operatorname{Minimals}_{\leq_\CaC}\bigl(S\setminus \{0\}\bigr)$. 
        \item $\operatorname{M}(S\cup\{x\})=\operatorname{M}(S)$.
    \end{enumerate}
\end{lemma}
\begin{proof}
    Let us denote $T=S\cup \{x\}$. In particular, $\CaH(S)=\CaH(T)\cup \{x\}$.
    
    \noindent We start by proving the first claim. If $y\in \operatorname{C}(T)$, then there exists $h\in \CaH(T)\subset \CaH(S)$ such that $y\leq_\CaC h$. That means $y\in \operatorname{C}(S)$. So, $\operatorname{C}(T)\subseteq \operatorname{C}(S)$. Conversely, if $y\in \operatorname{C}(S)$, then there exists $h\in \CaH(S)= \CaH(T)\cup \{x\}$ such that $y\leq_\CaC h$. If $h\in \CaH(T)$, then $y\in \operatorname{C}(T)$. Assuming $h=x$, by hypothesis there exists $h'\in \CaH(S)\setminus \{x\}$ such that $x \leq_\CaC h'$. In particular, $h'\in \CaH(T)$ and we have $y\leq_\CaC x\leq_\CaC h'$, which implies $y\in \operatorname{C}(T)$. So, $\operatorname{C}(S)\subseteq \operatorname{C}(T)$ and the first claim holds.
    
   \noindent Now we prove the second claim. Suppose $y\in \operatorname{Minimals}_{\leq_\CaC}\bigl(T\setminus \{0\})$. Hence, for all $z\in \CaC\setminus \{0\}$ such that $z<_\CaC y$, we have $z\in\CaH(T)\subset \CaH(S)$, that implies $y\in \operatorname{Minimals}_{\leq_\CaC}\bigl(S\setminus \{0\})$. Conversely, let $y\in \operatorname{Minimals}_{\leq_\CaC}\bigl(S\setminus \{0\})$ and assume $y\notin \operatorname{Minimals}_{\leq_\CaC}\bigl(T\setminus \{0\})$. This means that $x\leq_\CaC y$. Now, let $z \in \operatorname{I}_S(x)$. Since $z<_\CaC y$, by the minimality of $y$ in $S\setminus \{0\}$ we have $z=0$. This implies that $\operatorname{I}_S(x)=0$, which means $x\in \operatorname{M}(S)$, a contradiction. So, $y\in \operatorname{Minimals}_{\leq_\CaC}\bigl(T\setminus \{0\})$.    
   
   \noindent Finally, since $\operatorname{M}(S)=\{y\in \CaC\mid s\nleq_\CaC y\text{ for all }s\in \operatorname{Minimals}_{\leq_\CaC}(S\setminus \{0\})\}$, and applying the second claim, we conclude that $\operatorname{M}(T)=\operatorname{M}(S)$.
\end{proof}

The condition established in Lemma \ref{lemma:C-and-min-minus-x} can be reformulated as follows.

\begin{lemma}\label{lemma:technical}
    Let $S$ be a $\CaC$-semigroup and let $x\in \operatorname{SG}(S)$. Then,

  \begin{align*}
        x \notin \operatorname{M}(S) &\cup \operatorname{Maximals}_{\leq_\CaC}(\CaH(S))
        \notag\\
        \Leftrightarrow 
         x &\in \operatorname{C}(S \cup \{x\}) 
        \setminus 
        \operatorname{Minimals}_{\leq_\CaC}\bigl((S  \cup \{x\})\setminus \{0\}\bigr).
    \end{align*}
\end{lemma}
\begin{proof}
    
    Let $x\in \operatorname{SG}(S)$. First, observe that $x\in \operatorname{Minimals}_{\leq_\CaC}\bigl((S  \cup \{x\})\setminus \{0\}\bigr)$ if and only if for every $y \in \CaC\setminus\{0\}$ such that $y<_\CaC x$ we have that $y\in \CaH(S)$, and this is the definition of $x \in \operatorname{M}(S)$. Therefore, $x\notin \operatorname{M}(S)$ if and only if $x\notin \operatorname{Minimals}_{\leq_\CaC}\bigl((S  \cup \{x\})\setminus \{0\}\bigr)$. Now, suppose that $x\notin \operatorname{Maximals}_{\leq_\CaC}(\CaH(S))$. Then $x\in \operatorname{C}(S\cup\{x\})$, since $x\in \CaH(S)\subseteq \operatorname{C}(S)$. Conversely, suppose that $x\in \operatorname{C}(S\cup\{x\})$, that is, there exists $h\in \CaH(S\cup\{x\})=\CaH(S)\setminus\{x\}$ such that $x<_\CaC h$. So, $x\notin \operatorname{Maximals}_{\leq_\CaC}(\CaH(S))$.
\end{proof}

Let $S\in \mathcal{P}(k)\setminus \operatorname{EI}(k)$. In this case we know that $\operatorname{B}(S)\neq \emptyset$ and we define $\beta(S)=\max_{\preceq} \operatorname{B}(S)$. From Lemmas \ref{lemma:C-and-min-minus-x} and \ref{lemma:technical}, we present the following result. 

\begin{proposition}
    If $S\in\mathcal{P}(k)\setminus \operatorname{EI}(k)$, then $S\setminus\{\beta(S)\}\in \mathcal{P}(k)$.
  
\end{proposition}

\begin{proof}
    Let $S\in\mathcal{P}(k)\setminus \operatorname{EI}(k)$, and set $T=S\setminus\{\beta(S)\}$. Clearly, $T$ is a $\CaC$-semigroup since $\beta(S)$ is a minimal generator of $S$. By applying Lemma \ref{lemma:technical}, we deduce that $\beta(S)\notin\operatorname{M}(T) \cup \operatorname{Maximals}_{\leq_\CaC}(\CaH(T))$, since $\beta(S)\in \operatorname{SG}(T)$. Thus, the hypotheses of Lemma \ref{lemma:C-and-min-minus-x} are satisfied and since $S$ is primary positioned we conclude that $|\operatorname{I}_\CaC(k)|=|\operatorname{M}(S)|+|\operatorname{C}(S)|=|\operatorname{M}(T)|+|\operatorname{C}(T)|$. Hence, in order to conclude, it is sufficient to show that $T$ is $k$-positioned. In fact, if $h\in \CaH(T)$, then $h\in \CaH(S)\cup \{\beta(S)\}$. In particular, in the case $h=\beta(S)$, by definition we have $k-\beta(S)\in S\setminus \{\beta(S)\}= T$. While, since $S$ is $k$-positioned, if $h\in \CaH(S)$, then $k-h\in S=T\cup \{\beta(S)\}$. In this case, it is not possible $k-h=\beta(S)$, otherwise $k-\beta(S)\notin S$, that contradicts the definition of $\beta(S)$. So, $k-h\in T$ and we can conclude that $T$ is $k$-positioned.

\end{proof}

With the theoretical foundation established, we introduce the transform $\Psi_k: \mathcal{P}(k)\setminus \operatorname{EI}(k)\longrightarrow \mathcal{P}(k)$ defined by $\Psi_k(S)= S\setminus\{\beta(S)\}$. Since the set $\operatorname{B}(S)$ is finite for every $\CaC$-semigroup $S\in \mathcal{P}(k)\setminus \operatorname{EI}(k)$ and by Proposition \ref{prop:empty}, we deduce that there exists $n \in \mathbb{N}$, depending on $S$,  such that $\Psi_k^n(S) \in \operatorname{EI}(k)$. So, given any $T\in \operatorname{EI}(k)$ we can consider the set  
 \[
 \mathcal{P}_T(k) = \big\{ S \in \mathcal{P}(k) \mid \exists \,n \in \mathbb{N} \text{ such that } \Psi_k^n(S) = T \big\}.
 \]
 In this context, we define the graph $\operatorname{G}(\mathcal{P}_T(k))$, whose vertex set is the set $\mathcal{P}_T(k)$, and the pair $(S_1,S_2)\in\mathcal{P}_T(k)\times \mathcal{P}_T(k)$ is an edge is and only if  $\Psi_k(S_1)=S_2$. If $(S_1, S_2)$ is an edge, we say that $S_1$ is a child of $S_2$.

\begin{proposition}\label{Prop:rootedtree}
    For every $T\in \operatorname{EI}(k)$, the graph $\operatorname{G}(\mathcal{P}_T(k))$ is a tree, where $T$ is the root.
\end{proposition}
\begin{proof} 
Let $S\in \mathcal{P}_T(k)$ with $S\neq T$. So, there exists $n\in \mathbb{N}$ such that $\Psi_k^n(S)=T$. In particular, there exists a sequence of edges $(S_1,S_2)$, $(S_2, S_3), \ldots, (S_{n-1},S_n)$ such that $S_1=S$, $S_n=T$, and $S_{i+1}=\psi(S_i)$ for every $i\in \{1,\ldots,n-1\}$. This sequence is unique since the elements $\beta(S_i)$, for every $i\in \{1,\ldots,n-1\}$, are uniquely defined. In particular, we have a unique path of edges linking $S$ with $T$. Therefore, the graph is a tree with root $T$.
\end{proof}

Observe that, given any $S\in \mathcal{P}(k)$ the set of children of $S$ equals to
\[
\left\{S\cup\{x\}\mid x\in \operatorname{SG}(S) \text{ and } x=\beta(S\cup\{x\})\right\},
\]
and as a consequence of Lemma~\ref{lemma:technical} we obtain that if $x=\beta(S\cup \{x\})$, then $x\in \operatorname{SG}(S)$ and $x\notin \operatorname{M}(S)\cup \operatorname{Maximals}_{\leq_\CaC}(\CaH(S))$.  Now, our purpose is focused on determining the children of any primary positioned $\CaC$-semigroup $S$ without directly computing $S\cup \{x\}$ and $\beta(S\cup \{x\})$ for each $x\in \operatorname{SG}(S)$ and $x\notin \operatorname{M}(S)\cup \operatorname{Maximals}_{\leq_\CaC}(\CaH(S))$. To reach our aim we distinguish two cases: whether \( S \in \operatorname{EI}(k) \), i.e., when we are interested in computing the children of the root, or when we are concerned with the children of any other vertex. This will be the content of the next two subsections. 

\subsection{The children of a root}

Let $T\in \operatorname{EI}(k)$ and $x\in \operatorname{SG}(T)$. By the maximality condition of $\beta(T\cup \{x\})$, we might intuitively expect that elements $x\in\operatorname{SG}(S)$ satisfying $x \succ \frac{k}{2}$ would ensure that  $T \cup \{x\} $ is a child of $T$. We will analyse this case later. However, this approach leads to a natural question: Can exist a child $T\cup\{x\}$ of $T$ such that $x\in\operatorname{SG}(T)$ and $x\prec \frac{k}{2}$? To study this particular scenario, we present the following two examples.

\begin{example}
Let $T$ be the primary positioned $\mathbb{N}^2$-semigroup  for $k=(6,6)$, defined in Figure \ref{fig:3k}. Take $x=(2,2)\in\operatorname{SG}(T)$ and consider the $\CaC$-semigroup $S=T\cup\{(2,2)\}$. Observe that $\operatorname{B}(S)=\{x\}=\{\beta(S)\}$. Therefore, $S$ is a child of $T$ and $x\prec\frac{k}{2}$.   
    \begin{figure}[htbp]
        \centering
        \includegraphics[width=0.5\linewidth]{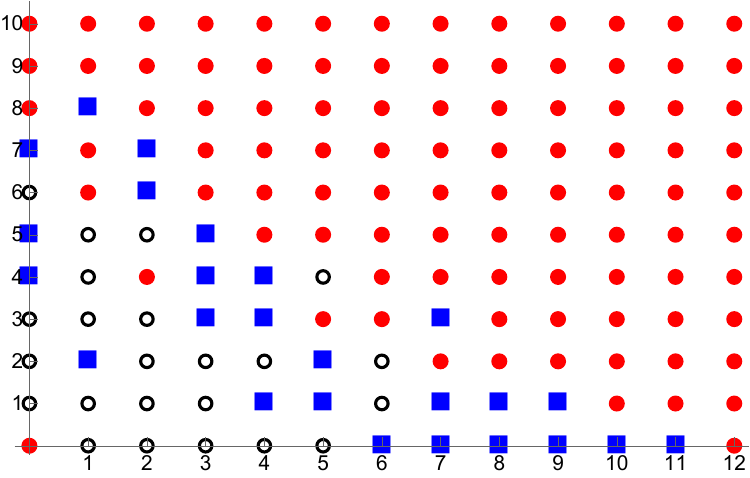}
        \captionsetup{width=0.58\textwidth}
        \caption{Primary positioned $\mathbb{N}^2$-semigroup $T$ for $k=(6,6)$, $\circ\in\CaH(T)$;  \textcolor{red}{$\bullet$} $\in T$; \textcolor{blue}{\scalebox{0.6}{$\blacksquare$}} $\in \operatorname{msg}(T)$}
        \label{fig:3k}
    \end{figure}
\end{example}

\begin{example}
Consider the primary positioned $\mathbb{N}^2$-semigroup $T$ for $k=(4,8)$, graphically represented in Figure  \ref{fig:4k}. Fixed $x = (1,2) \in \operatorname{SG}(T)$, and define the $\CaC$-semigroup $S = T \cup {(1,2)}$. Computing $\operatorname{B}(S) = \{x\} = \{\beta(S)\}$. So, $S$ is a child of $T$ and $x\prec\frac{k}{2}$.   
\begin{figure}[htbp]
        \centering
        \includegraphics[width=0.5\linewidth]{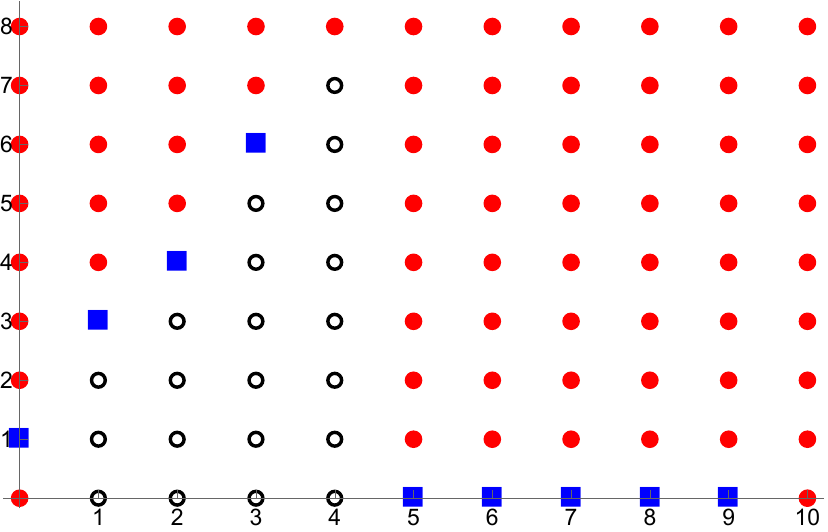}
        \captionsetup{width=0.58\textwidth}
        \caption{Primary positioned $\mathbb{N}^2$-semigroup $T$  for $k=(4,8)$, $\circ\in\CaH(T)$;  \textcolor{red}{$\bullet$} $\in T$; \textcolor{blue}{\scalebox{0.6}{$\blacksquare$}} $\in \operatorname{msg}(T)$}
        \label{fig:4k}
\end{figure}
\end{example}

The examples above help us formulate the next result. Apart from the condition $x\succ \frac{k}{2}$, the special gaps described in them are the particular cases of the following characterization of the children of a $\CaC$-semigroup belonging to $\operatorname{EI}(k)$. 

\begin{theorem}\label{Thrm:FirstStep}
    Let $T\in \operatorname{EI}(k)$. Then, all the children of $T$ are the $\CaC$-semigroups $T\cup\{x\}$ such that $x\in \operatorname{SG}(T)$, $x\notin\operatorname{M}(T)\cup\operatorname{Maximals}_{\leq_\CaC}(\CaH(T))$ and $x\succ \frac{k}{2}$, $3x=k$ or $4x=k$.
\end{theorem}
\begin{proof}
Consider the set $G=\{g\in SG(T) \mid g=\beta(T\cup \{g\})\}$. Let $S$ be a child of $T$, by definition, $S=T\cup\{g\}$ where $g\in G$. Since $g\in \operatorname{B}(T\cup \{g\})$, we have $g\in \operatorname{C}(T\cup \{g\})\setminus \{\operatorname{Minimals}_{\leq_\CaC}\bigl((T  \cup \{g\})\setminus \{0\}\bigr)$. By Lemma~\ref{lemma:technical}, $g\notin \operatorname{M}(T)\cup \operatorname{Maximals}_{\leq_\CaC}(\CaH(T))$. So, we need to prove that $g\succ \frac{k}{2}$, $3g=k$, or $4g=k$. Since $g=\beta( T\cup\{g\})$ we have that  $k-g\in T$. Let us prove that $k-g\in \operatorname{msg}(T)$. Suppose that $k-g=n+t$ for some $n\in \operatorname{msg}(T)$ and $t\in T\setminus \{0\}$. In particular, $k=(g+n)+t$ and $k=(g+t)+n$ are different expressions of $k$ in $T$. By Proposition~\ref{Prop:symm&k-positioned}, we obtain that $T$ is a PEPSY-semigroup, and by applying Proposition \ref{Prop:PEPSY&k-positioned}, it follows that $g+n=g+t$ or $g+n=n$. Trivially, $g+n=n$ is not possible since $g\neq 0$. So, $n=t$ and $k=(g+n)+n$, with $n=\frac{k}{2}$ and $g+n=\frac{k}{2}$, which is impossible. Therefore, $k-g\in \operatorname{msg}(T)$.
 Moreover, let us see that $k-g\in \operatorname{C}(T\cup \{g\})\setminus \operatorname{Minimals}_{\leq_\CaC}\bigl((T  \cup \{g\})\setminus \{0\}\bigr)$. In fact, if $k-g\notin \operatorname{C}(T\cup \{g\})$, then $g=k-(k-g)\in \operatorname{M}(T\cup \{g\})$, since $T$ is primary positioned. If $k-g\in \operatorname{Minimals}\bigl((T  \cup \{g\})\setminus \{0\}\bigr)$, then $g\in \operatorname{Maximals}_{\leq_\CaC}(\CaH(T\cup \{g\}))$, by Lemma~\ref{lemma:min-max}. In both cases we obtain $g\in \CaH(T\cup\{g\})$, that is a contradiction. Suppose now $\operatorname{msg}(T\cup \{g\})=\{n_1,\ldots,n_r,g\}$, by \cite[Lemma 3]{Wilf-Vigneron} we have that $\{n_1,\ldots,n_r,g+n_1,\ldots,g+n_r, 2g, 3g\}$ is a system of generators of $T$. Assume $k-g=n_i$ for some $i\in \{1,\ldots,r\}$. It follows that $k-g\in \operatorname{msg}(T\cup \{g\})$. In particular, $k-g\in \operatorname{B}(T\cup \{g\})$ and by the maximality of $g$ we obtain $g\succ k-g$. By the proof of Lemma~\ref{lem:k/2}, we obtain $g\succ \frac{k}{2}$. Now assume $k-g=g+n$ with $n\in \{n_1,\ldots,n_r, g, 2g\}$. If $k=2g+n_i$ for some $i\in \{1,\ldots,r\}$, then $2g=n_i=\frac{k}{2}$ since $2g, n_i\in T$ and $T\in \operatorname{EI}(k)$, contradicting $n_i\in \operatorname{msg}(T\cup \{g\})$. In the other cases, we obtain $k=3g$ or $k=4g$, and we conclude.

Conversely, suppose that $x\in \operatorname{SG}(T)$ $x\notin\operatorname{M}(T)\cup\operatorname{Maximals}_{\leq_\CaC}(\CaH(S)))$, and $x\succ \frac{k}{2}$, $3x=k$ or $4x=k$, and let us prove that $x\in G$. So, we have to show that $x=\beta(T\cup \{x\})$. Since $x\in \operatorname{SG}(T)$ we know that $T\cup \{x\}$ is a $\CaC$-semigroup and $x\in \operatorname{msg}(S)$. Furthermore, $x\ne \frac{k}{2}$. This fact, together with the property of being primary positioned for $k$, means that $k-x\in T\subset T\cup \{x\}$ since $x\in \CaH(T)$.  By applying Lemma \ref{lemma:technical} we deduce that $x\in \operatorname{B}(T\cup\{x\})$. Suppose that $y=\beta(T\cup \{x\})$ with $y\neq x$. Since $y\in \operatorname{B}(T\cup \{x\})$ we have $y\neq \frac{k}{2}$ and $k-y\in T\cup \{x\}$. Assume first that $k-y \in T$. Since $y\in \operatorname{msg}(T\cup \{x\})\setminus \{x\}$, we obtain that $y\in T$. In particular, $k=y+t$ with $y,t\in T\setminus \{0,\frac{k}{2}\}$, which is a contradiction of $T\in \operatorname{EI}(k)$ by Propositions~\ref{Prop:symm&k-positioned} and \ref{Prop:PEPSY&k-positioned}. If $k-y=x$ we distinguish two cases. On one hand, we assume that $\lambda x=k$ with $\lambda\in \{3,4\}$, then $y=(\lambda-1)x$ and therefore $y\notin\operatorname{msg}(T\cup\{x\})$, a contradiction. On the other hand, if $x\succ \frac{k}{2}$, then by the proof of  Lemma \ref{lem:k/2} we know that $y=k-x\prec \frac{k}{2}\prec x$, contradicting the maximality of $y$. So,  $x=\beta(T\cup \{x\})$ and $x\in G$. Hence, we conclude that $T\cup\{x\}$ is a child of $T$.
\end{proof}

\subsection{The children of a non-root Vertex}
Now our aim is to provide a result analogous to Theorem~\ref{Thrm:FirstStep} for the $\CaC$-semigroups $S\in \mathcal{P}(k)\setminus \operatorname{EI}(k)$. First, we give some preliminary Lemmas, which highlight the role played by $\beta(S)$ and help us to formulate the general result.

\begin{lemma}\label{lem:x>beta}
Let $S\in \mathcal{P}(k)\setminus \operatorname{EI}(k)$. Suppose $x\in \operatorname{SG}(S)$ such that 
$x\notin \operatorname{M}(S)\cup \operatorname{Maximals}_{\leq_\CaC}(\CaH(S))$ and $x\succ \beta(S)$. Then $S\cup \{x\}$ is a child of $S$.
\end{lemma}

\begin{proof}
To obtain our result, it is sufficient to prove that $x=\beta(S\cup \{x\})$. Since $S$ is $k$-positioned and $x$ is a gap of $S$, we know that $x\neq \frac{k}{2}$ and $k-x\in S\subseteq S\cup \{x\}$. Trivially, $x\in \operatorname{msg}(S\cup \{x\})$ and by Lemma~\ref{lemma:technical}, we have $ x\in \operatorname{C}(S\cup\{x\})\setminus\operatorname{Minimals}_{\leq_\CaC}\bigl((S  \cup \{x\})\setminus \{0\}\bigr)$. So, $x\in \operatorname{B}(S\cup \{x\})$. Suppose that $x\neq \beta(S\cup \{x\})$, it follows that $x\prec \beta(S\cup \{x\})$. By definition we have $\beta(S\cup \{x\})\in \operatorname{msg}(S\cup \{x\})\subseteq \operatorname{msg}(S)\cup\{x\}$. Since we assumed $\beta(S\cup \{x\})\neq x$, we deduce $\beta(S\cup \{x\})\in \mathrm{msg}(S)$. Considering also Lemma~\ref{lemma:C-and-min-minus-x}, we have $\beta(S\cup \{x\}) \in \operatorname{C}(S\cup\{x\})\setminus\operatorname{Minimals}_{\leq_\CaC}\bigl((S  \cup \{x\})\setminus \{0\}\bigr)= \operatorname{C}(S)\setminus\operatorname{Minimals}_{\leq_\CaC}(S\setminus\{0\})$. 

Moreover, $k-\beta(S\cup \{x\})\in S\cup \{x\}$. In particular, we have two possibilities, $k-\beta(S\cup \{x\})=x$ or $k-\beta(S\cup \{x\})\in S$. In the first case, observe that if $T\in \operatorname{EI}(k)$ is the semigroup such that $S\in \mathcal{P}_T(k)$, that is, $T$ is the root of the graph having $S$ as a vertex, then $x,\beta(S\cup \{x\})\in \CaH(T)$ and since $T$ is $k$-positioned then $k-\beta(S\cup \{x\})\in T$. This means the equality $k-\beta(S\cup \{x\})=x$ is not possible. Therefore, $k-\beta(S\cup \{x\})\in S$ and this implies $\beta(S\cup \{x\})\in \operatorname{B}(S)$. As a consequence, $\beta(S\cup \{x\})\preceq \beta(S)\prec x$, a contradiction. So, $x= \beta(S\cup \{x\})$ and we conclude.
\end{proof}

Nevertheless, if $S\in \mathcal{P}(k)\setminus \operatorname{EI}(k)$ and $x\in \operatorname{SG}(S)$, then it is possible to obtain that $S\cup \{x\}$ is a child of $S$ and $x\prec \beta(S)$, as shown in the following example.

\begin{example}
Let $S$ be the GNS in $\mathbb{N}^2$ having set of gaps:
\begin{align*}
\CaH(S)=\lbrace &(0,1),(0,2),(0,3),(0,4),(0,5),(1,0),(1,1),(1,2),(1,3),\\
& (1,4),(1,5),(2,2),(2,3),(2,4),
(2,5),(3,0),(3,3),(3,4),(4,5)\rbrace.
\end{align*}
The $\mathbb{N}^2$-semigroup $S$ is primary positioned for $k=(6,5)$. It is also possible to check that $\operatorname{SG}(S)=\{$(1, 5), (2, 2), (2, 3), (3, 0), (3, 3), (3, 4), (4, 5)$\}$. Moreover, the set of minimal generators is:
\begin{align*}
    \operatorname{mgs}(S) = \lbrace 
    & (0,6), (0,7), (0,9), (0,8), (0,10), (0,11),(2,0), \\
    &  (2,1), (1,6),(1,7), (3,1), (1,8), (3,2),(1,9), (1,10),  \\
    & (1,11), (3,5), (4,3), (4,4), (5,0), (5,4) 
    \rbrace.
\end{align*}
Assume $\preceq$ is the graded lexicographic order. In this case it is possible to compute that $\operatorname{B}(S)=\{(2,1),(4,4)\}$, $\beta(S)=(4,4)$ and the children of $S$ are $S\cup \{(2,2)\}$ and $S\cup \{(2,3)\}$.
\end{example}

The example above satisfies the following necessary condition for $S\cup\{x\}$ to be a child of $S$ with $x\prec \beta(S)$.

\begin{lemma}\label{lem:x<beta}
Let $S\in \mathcal{P}(k)\setminus \operatorname{EI}(k)$ and $x\in \operatorname{SG}(S)$ such that $x\prec \beta(S)$. If $x=\beta(S\cup \{x\})$, then for all $y\in \operatorname{B}(S)$ with $x\prec y$ it holds $y-x\in S$ or $y=2x$.
\end{lemma}
\begin{proof}
Let $y\in \operatorname{B}(S)$ with $x\prec y$. First, consider that $y\in \operatorname{msg}(S)$. Suppose that $y\neq 2x$ and $y-x\notin S$. In this case, it is not difficult to argue that $y\in \operatorname{msg}(S\cup \{x\})$. Considering $x=\beta(S\cup \{x\})$, by definition and Lemma~\ref{lemma:technical} we have $x\notin \operatorname{M}(S)\cup \operatorname{Maximals}_{\leq_\CaC}(\CaH(S))$. Whence, by Lemma~\ref{lemma:C-and-min-minus-x}, $\operatorname{C}(S)\setminus\operatorname{Minimals}_{\leq_\CaC}(S\setminus\{0\}) = \operatorname{C}(S\cup\{x\})\setminus\operatorname{Minimals}_{\leq_\CaC}\bigl((S  \cup \{x\})\setminus \{0\}\bigr)$. In particular, we deduce $y\in \operatorname{C}(S)\setminus\operatorname{Minimals}_{\leq_\CaC}(S\setminus\{0\}) = \operatorname{C}(S\cup\{x\})\setminus\operatorname{Minimals}_{\leq_\CaC}\bigl((S  \cup \{x\})\setminus \{0\}\bigr)$. Furthermore, $y\neq \frac{k}{2}$ and $k-y\in S\subseteq S\cup \{x\}$. So, putting together all these conditions, we have $y\in \operatorname{B}(S\cup \{x\})$. As a consequence, $y\preceq \beta(S\cup \{x\})=x$, contradicting $x\prec y$. We conclude that $y=2x$ or $y-x\in S$. 
\end{proof}

We can gather the previous conditions to obtain the following characterization.

\begin{theorem}
Let $S\in \mathcal{P}(k)\setminus \operatorname{EI}(k)$. Then, all the children of $S$ are the $\CaC$-semigroups $S\cup\{x\}$ such that $x\in \operatorname{SG}(S)$, $x\notin\operatorname{M}(S)\cup\operatorname{Maximals}_{\leq_\CaC}(\CaH(S))$, and one of the following holds:
\begin{itemize}
\item $x\succ \beta(S)$.
\item $x\prec \beta(S)$ and for all $y\in \operatorname{B}(S)$ with $x\prec y$ it holds $y-x\in S$ or $y=2x$.
\end{itemize}
\end{theorem}

\begin{proof}
Consider the set $G=\{g\in \operatorname{SG}(S) \mid g=\beta(S\cup\{g\})\}$ and let $S\cup\{g\}$ be a child of $S$, with $g\in G$. As shown in the proof of Theorem~\ref{Thrm:FirstStep}, since $g \in \operatorname{B}(S\cup \{g\})$, by Lemma~\ref{lemma:technical}, we know that $g \notin \operatorname{M}(S) \cup \operatorname{Maximals}_{\leq_\CaC}(\CaH(S))$. If $g\succ \beta(S)$ we have nothing to prove, while if $g\prec \beta(S)$ we obtain our claim by Lemma~\ref{lem:x<beta}. So, the necessity is proved. In order to prove sufficiency, suppose that $x\in \operatorname{SG}(S)$ and $x\notin\operatorname{M}(S)\cup\operatorname{Maximals}_{\leq_\CaC}(\CaH(S)))$. If $x\succ \beta(S)$ we obtain that $x=\beta(S\cup \{x\})$ by Lemma~\ref{lem:x>beta}, and thus $x\in G$. So, assume $x\prec \beta(S)$ and for all $y\in \operatorname{B}(S)$ with $x\prec y$ it holds $y-x\in S$ or $y=2x$. Suppose that $x\notin G$, thus $x\neq \beta(S\cup \{x\})$. Observe that $x\in \operatorname{msg}(S\cup \{x\})$ and since $S$ is $k$-positioned we have $x\neq \frac{k}{2}$ and $k-x\in S\subset S\cup \{x\}$. By applying Lemma~\ref{lemma:technical} we also obtain that $x\in \operatorname{C}(S\cup\{x\})\setminus\operatorname{Minimals}_{\leq_\CaC}\bigl((S\cup\{x\})\setminus\{0\}\bigr)$. So, $x\in \operatorname{B}(S)$ and by our assumption $x\prec \beta(S\cup \{x\})$. Let us denote $z=\beta(S\cup \{x\})$. It is not difficult to see that $\operatorname{msg}(S\cup \{x\})\subseteq \operatorname{msg}(S)\cup \{x\}$. In particular, we deduce $z\in \operatorname{msg}(S)$. By definition, $k-z\in S\cup \{x\}$ and $z\neq \frac{k}{2}$. It is not possible $k-z=x$. In fact, if $T\in \operatorname{EI}(k)$ is the semigroup such that $S\in \mathcal{P}_T(k)$ then $x,z\in \CaH(T)$. As a consequence, since $T$ is $k$-positioned, we have $k-z\in T$, whence $k-z\neq x$. Therefore, $k-z\in S$. Furthermore, by definition  $z\in \operatorname{C}(S\cup\{x\})\setminus\operatorname{Minimals}_{\leq_\CaC}\bigl((S\cup\{x\})\setminus\{0\}\bigr)$, and since $x\notin \operatorname{M}(S)\cup \operatorname{Maximals}_{\leq_\CaC}(\CaH(S))$, by Lemma~\ref{lemma:C-and-min-minus-x}, we have $\operatorname{C}(S\cup\{x\})\setminus\operatorname{Minimals}_{\leq_\CaC}\bigl((S\cup\{x\})\setminus\{0\}\bigr)=\operatorname{C}(S)\setminus\operatorname{Minimals}_{\leq_\CaC}(S\setminus\{0\})$. Therefore, considering all these conditions together, we obtain $z\in \operatorname{B}(S)$. Since $x\prec z$, by hypothesis, we have $z-x\in S$ or $z=2x$ and this implies that $z\notin \operatorname{msg}(S\cup \{x\})$, which leads to a contradiction. We conclude $x=\beta(S\cup \{x\})$.
\end{proof}

\section{A procedure for computing the set $\mathcal{P}(k)$}\label{Sec6}

This section is devoted to the computation of the set $\mathcal{P}(k)$ using as tools the arguments presented in Section~\ref{section-tree}. Additionally, we provide an algorithm which allows us to show how the set  $\mathcal{P}(k)$ can be arranged in a graph. We end the work by illustrating the proposed procedures with some selected examples.
For any $k\in \CaC$, we consider
\[
G(\mathcal{P}(k))=\bigcup_{T\in \operatorname{EI}(k)}G(\mathcal{P}_T(k)).
\]
\begin{proposition}
The graph $\operatorname{G}(\mathcal{P}(k))$ is a forest, whose set of roots is $\operatorname{EI}(k)$.
\end{proposition}
\begin{proof}
By Proposition \ref{Prop:rootedtree}, we know that each $G(\mathcal{P}_T(k))$ is a rooted tree. By the uniqueness of $\beta(S)$ for any $S\in \mathcal{P}(k)$, we deduce that 
$\cup_{T\in \operatorname{EI}(k)}\mathcal{P}_T(k)$ are disjoint sets. Hence, $\operatorname{G}(\mathcal{P}(k))$ is a forest and $\operatorname{EI}(k)$ is the set of roots, as these are the roots of the individual trees in the forest.
\end{proof}

In order to compute the set $\mathcal{P}(k)$ we establish the following steps:
\begin{enumerate}
\item Compute first the set $\operatorname{EI}(k)$.
\item For every $T\in \operatorname{EI}(k)$, compute the set $\mathcal{P}_T(k)$.
\item Finally, $\mathcal{P}(k)=\bigsqcup_{T\in \operatorname{EI}(k)} \mathcal{P}_T(k)$.
\end{enumerate}

For computing $\operatorname{EI}(k)$, consider the set $\mathcal{I}(k)$ of all irreducible $\CaC$-semigroup having Frobenius element equal to $k$. This set can be computed using \cite[Algorithm 1]{garcia-on-some}, by Remark \ref{rem:EI(k)} we count the set $\operatorname{EI}(k)$. If $T\in \operatorname{EI}(k)$, we can compute the set $\mathcal{P}_T(k)$ using the structure of the tree described in Section~\ref{section-tree}. We can sum up our suggested procedure in Algorithm~\ref{alg1}.

\begin{algorithm}
\caption{Algorithm for computing the set $\mathcal{P}_T(k)$} \label{alg1}
\DontPrintSemicolon
\KwData{A positive integer cone $\CaC$,  a $\CaC$-semigroup $T\in \operatorname{EI}(k)$, $k\in \CaC$,  and a term order $\preceq$.}
\KwResult{$\mathcal{P}_T(k)$}
 $A:=\{T\}$.\;
 $B:=\emptyset$.\;
 $C:=\{x\in \operatorname{SG}(T)\mid x\notin \operatorname{M}(T)\cup \operatorname{Maximals}_{\leq_\CaC}(\CaH(T))\}.$\;
 $C_1:=\{x\in C\mid x\succ \frac{k}{2}\text{ or }3x=k \text{ or } 4x=k\}.$\;
  \If{$C_1=\emptyset$}{
       \Return $\{T\}$\;
       }
   \For{$x\in C_1$}
                {$B:=B\cup \{T\cup \{x\}\}.$\;
   $\beta(T\cup \{x\}):=x$ (save this value with $T\cup \{x\}$).\;             
                }
       $A:=A\cup B$.\; 
       $B_1:=B$.\;
       \While{$B_1\neq \emptyset$}{
       $B_2:=\emptyset$.\;
       \For{$S\in B_1$}{
       $B:=\emptyset$.\;
       $C:=\{x\in \operatorname{SG}(S)\mid x\notin \operatorname{M}(S)\cup \operatorname{Maximals}_{\leq_\CaC}(\CaH(S))\}.$\;
       $C_1:=C\cap \{x\succ \beta(S)\}.$\;
       \For{$x\in C_1$}{
       $B:=B\cup \{S\cup \{x\}\}.$\;
       $\beta(S\cup \{x\}):=x$ (save this value with $S\cup \{x\}$).\;
       }
       \If{$C\setminus C_1\neq \emptyset$}{
       Compute $\operatorname{B}(S)$.\;
       \For{$x\in C\setminus C_1$}{
       \If{for all $y\in \operatorname{B}(S)$ with $x\prec y$ it holds $y-x\in S$ or $y=2x$}{$B:=B\cup \{S\cup \{x\}\}.$\;
       $\beta(S\cup \{x\}):=x$ (save this value with $S\cup \{x\}$).\;
       }
       }
       } 
       $A:=A\cup B$.\;
       $B_2:=B_2\cup B.$\;      
       }
       $B_1:=B_2$.\;
       }
       \Return $A$\;
\end{algorithm}

\begin{remark}
The instructions in lines $9,20,26$ of Algorithm~\ref{alg1} are coherent with our arguments. In fact, we know that if $S\cup \{x\}$ is a child of $S$, then $x=\beta(S\cup \{x\})$. These instructions can be useful since we need $\beta(S)$ in line 17, and it is not necessary to compute it if we save this value in the previous steps. 
\end{remark}

The following  two examples illustrate how Algorithm \ref{alg1} works. In particular, the first example emphasizes that for smaller values of $k$, the  forest $\operatorname{G}(\mathcal{P}(k))$ is relatively simple since the set $\mathcal{P}(k)$ has a few elements. If $k$ increases, more vertices emerge, making it more suitable to illustrate with a single tree rather than the entire forest. To this purpose, we present the second example.

\begin{example}
Consider the cone $\mathcal{C}=\mathbb{N}^2$ and the set $\operatorname{EI}((2,3))$. It is the set of all generalized numerical semigroups $S=T\cup \{(2,3)\}$, such that $T$ is symmetric and $\operatorname{F}(T)=(2,3)$. It is possible to compute that this set consists of the following monoids:
\begin{enumerate}
\item $S_1=\mathbb{N}^2\setminus \{(0,1),(0,2),(0,3),(1,0),(1,1)\}$,
\item $S_2=\mathbb{N}^2\setminus \{(0,1),(0,2),(0,3),(1,0),(1,2)\}$,
\item $S_3=\mathbb{N}^2\setminus \{(0,1),(0,2),(0,3),(1,1),(1,3)\}$,
\item $S_4=\mathbb{N}^2\setminus \{(0,1),(0,2),(1,0),(1,1),(2,0)\}$,
\item $S_5=\mathbb{N}^2\setminus \{(0,1),(0,3),(1,0),(1,1),(2,1)\}$,
\item $S_6=\mathbb{N}^2\setminus \{(0,1),(0,2),(0,3),(1,2),(1,3)\}$,
\item $S_7=\mathbb{N}^2\setminus \{(0,1),(0,2),(1,0),(1,2),(2,0)\}$,
\item $S_8=\mathbb{N}^2\setminus \{(0,1),(0,3),(1,0),(1,2),(2,1)\}$,
\item $S_9=\mathbb{N}^2\setminus \{(0,1),(0,3),(1,1),(1,3),(2,1)\}$,
\item $S_{10}=\mathbb{N}^2\setminus \{(0,1),(1,0),(1,1),(2,0),(2,1)\}$,
\item $S_{11}=\mathbb{N}^2\setminus \{(0,1),(1,0),(1,2),(2,0),(2,1)\}$,
\item $S_{12}=\mathbb{N}^2\setminus \{(1,0),(1,1),(2,0),(2,1),(2,2)\}$.
\end{enumerate}
By Algorithm \ref{alg1}, we obtain that if $S\in \operatorname{EI}((2,3))$ and $S\neq S_6$, then $S$ has no children. The monoid $S_6$ has only one child, that is the semigroup $S_{61}=S\cup \{(1,2)\}$, and $S_{61}$ has no children. So, for $\mathcal{C}=\mathbb{N}^{2}$ we have $\mathcal{P}((2,3))=\operatorname{EI}((2,3))\cup \{S_{61}\}$.  
\end{example}

For completeness, as previously mentioned, we  provide an example of the construction of a specific tree with a larger number of vertices. 

\begin{example}
Let $\mathcal{C}$ be the cone in $\mathbb{N}^2$ spanned by the set $\{(1,0),(1,1)\}$. Let $T$ be the $\mathcal{C}$-semigroup having the following set of gaps:
\begin{align*}
\CaH(T)=\lbrace &(1,1),(2, 2),(3,3),(4,4),(5,5),(1,0),(2,1), (3,2), (4,3), (5,4),\\ 
& (6,5), (4,2), (5,3), (6,4), (7,5), (3,0), (6,3), (7,4), (8,4), (9,5)\rbrace.
\end{align*}

Observe that $T$ is $(11,5)$-positioned and $T\setminus \{(11,5)\}$ is symmetric. So, $T\in \operatorname{EI}((11,5))$. 
We now construct the tree whose root is $T$. By Algorithm \ref{alg1} the monoid $T$ has the following children:
 $$T_1=T\cup \{(8,4)\},\qquad T_2=T\cup \{(7,4)\},\qquad
 T_3=T\cup \{(6,3)\}.$$
$T_1$ has the following children:
 $$ T_{11}=T_1 \cup \{(5,3)\}, \qquad T_{12}=T_1\cup \{(4,2)\}.$$
The children of $T_{11}$ are
$$ T_{111}=T_{11} \cup \{(7,4)\}, \qquad T_{112}=T_{11}\cup \{(6,3)\}.$$ 
Continuing the construction from the above monoids, $T_{111}$ has no children, while $T_{112}$ has only one child, that is $T_{1121}=T_{112}\cup \{(7,4)\}$, which has no further descendants.

\noindent The children of $T_{12}$ are these:
$$ T_{121}=T_{12} \cup \{(7,4)\}, \qquad T_{122}=T_{12}\cup \{(6,3)\}.$$ 
Proceeding  with the above monoids, $T_{121}$ has no children, while $T_{122}$ has only one child, that is $T_{1121}=T_{112}\cup \{(7,4)\}$, which have no children.

\noindent For $T_2$, there is a single child:
\[
T_{21} = T_2 \cup \{(8,4)\},
\]
which has no further descendants.

\noindent Finally, $T_3$ has two children:
$$ T_{31}=T_3\cup \{(8,4)\},\qquad T_{32}=T_3\cup \{(7,4)\}.$$
From above, $T_{31}$ has no children, while $T_{32}$ has only one child, that is $T_{321}=T_{32}\cup \{(8,4)\}$, which have no children.

\noindent The picture of tree having $T$ as a root is depicted in Figure~\ref{fig:tree}.

\begin{figure}[ht]
\centering
\begin{tikzpicture}
\tikzset{level distance=3em}
\Tree
        [.$T$  
        [.$T_1$ [.$T_{11}$ $T_{111}$ [.$T_{112}$ $T_{1121}$ ] ][.$T_{12}$ $T_{121}$ [.$T_{122}$ $T_{1221}$ ] ]
        ]    
        [.$T_2$ $T_{21}$
        ]  
        [.$T_3$ $T_{31}$ [.$T_{32}$ $T_{321}$ ]
        ]        
        ]
\end{tikzpicture}
\caption{The rooted tree $G(\mathcal{P}_T((11,5))$}
\label{fig:tree}
\end{figure}
\end{example}

\subsection*{Acknowledgements}
All authors thank the referee for her/his useful remarks and comments.

\subsection*{Funding}

The first author acknowledges support of INDAM-GNSAGA. The second author is partially supported by by grant PID2022-138906NB-C21 funded by MICIU/AEI/10.13039/501100011033 and by ERDF/EU.
Consejería de Universidad, Investigación e Innovación de la Junta de Andalucía project ProyExcel\_00868 and research group FQM343.

\subsection*{Author information}
C. Cisto.  Dipartimento di Scienze Matematiche e Informatiche, Scienze Fisiche e Scienze della Terra, 
Universit\`{a} di Messina, E-98166 Messina, (Italy).
E-mail: carmelo.cisto@unime.it.

\medskip
\noindent
R. Tapia-Ramos. Departamento de Matem\'aticas, Universidad de C\'adiz, E-11406 Jerez de la Frontera (C\'{a}diz, Spain).
E-mail: raquel.tapia@uca.es.

\subsection*{Data Availability}
The authors confirm that the data supporting some findings of this study are available within it.

\subsection*{Conflict of Interest}
The authors declare no conflict of interest.

\bibliographystyle{plain}
\bibliography{biblio.bib}

\end{document}